\documentclass[11pt,reqno]{amsart}

\usepackage[T1]{fontenc}
\usepackage[utf8]{inputenc}
\usepackage{lmodern}
\usepackage{microtype}
\usepackage[a4paper,margin=1in]{geometry}
\usepackage{amsmath,amssymb,amsthm,mathtools}
\usepackage{xcolor}
\usepackage[colorlinks=true,linkcolor=blue,citecolor=blue,urlcolor=blue]{hyperref}
\usepackage[capitalize,nameinlink,noabbrev]{cleveref}

\usepackage{tikz}
\usetikzlibrary{arrows.meta,decorations.pathreplacing,decorations.markings,calc}

\theoremstyle{plain}
\newtheorem{theorem}{Theorem}[section]

\newtheorem{definition}{Definition}[section]
\newtheorem{inputthm}[theorem]{Theorem}
\newtheorem{lemma}[theorem]{Lemma}
\newtheorem{proposition}[theorem]{Proposition}
\newtheorem{corollary}[theorem]{Corollary}

\theoremstyle{remark}
\newtheorem{remark}[theorem]{Remark}

\newcommand{\Z}{\mathbb{Z}}
\newcommand{\R}{\mathbb{R}}

\newcommand{\Prim}{\Z^2_{\mathrm{prim}}}
\newcommand{\ProjPrim}{\mathcal{P}}

\usepackage{comment}
\newif\ifshowrevisioncomments
\showrevisioncommentstrue
\ifshowrevisioncomments

\else
\excludecomment{bluecomment}
\fi

\DeclareMathOperator{\height}{ht}
\DeclareMathOperator{\Turn}{Turn}
\DeclareMathOperator{\Len}{Len}
\DeclareMathOperator{\arccosh}{arccosh}

\title[Stable-norm boundary geometry on once-punctured tori]{Stable-norm boundary geometry on once-punctured tori}

\author[N. M. Doan, X. Li, and V. Nguyen]{Nhat Minh Doan}
\address{Institute of Mathematics, Vietnam Academy of Science and Technology, Hanoi, Vietnam}
\email{dnminh@math.ac.vn}
\urladdr{https://sites.google.com/view/dnminh}

\author{Xiaobin Li}
\address{School of Mathematics, Southwest Jiaotong University,
West zone, High-tech district, Chengdu, Sichuan 611756, China}
\email{lixiaobin@home.swjtu.edu.cn}
\author{Van Nguyen}
\address{Institute of Mathematics, Vietnam Academy of Science and Technology, Hanoi, Vietnam}
\email{vannguyen1112203@gmail.com}
\date{}



\begin{document}

\begin{abstract}
We use normal-turn estimates to study the local and global geometry of the boundaries of stable-norm
balls $B_X$ for complete finite-area hyperbolic once-punctured tori $X$.
We show that the boundary $\partial B_X$ is a convex-geometric detector of
exponential Diophantine approximation: rational directions give genuine
corners with exponentially small exterior angles, while at an irrational
direction $\beta$ the graph-flatness order admits an explicit formula in terms of the exponential rate at which
rational directions approach $\beta$ and the $\ell^\infty$-radius of $B_X$ in the projective
direction $\beta$. Thus irrational directions are not uniformly flat to infinite
order, correcting the McShane--Rivin local picture. We also determine all
possible irrational flatness orders and the size of the corresponding level
sets.
The computation of the rational corners unexpectedly yields a new
McShane-type identity for the area of the dual body $B_X^*$:
$$
\operatorname{Area}(B_X^*)
=
4\sum_{\gamma}
\ell_X(\gamma)
\log\coth\left(\frac{\ell_X(\gamma)}{2}\right),
$$
where the sum is over unoriented simple closed geodesics on $X$. As another global
application, the normal-turn estimates also give a logarithmic-square bound for the number of integer points on the
boundary of each dilated norm ball. Consequently, the number of simple closed geodesics of length exactly
$L\geq2$ is at most $C_X(\log L)^2$. For the modular torus this yields
$$
\#\lambda_M^{-1}(m)\leq C(\log\log(3m))^2
$$
for every Markoff number $m$, improving the previous logarithmic bounds for
Markoff fibers.
\end{abstract}

\dedicatory{Dedicated to Greg McShane, with admiration and gratitude, on the occasion of his sixtieth birthday}

\maketitle
\markboth{N. M. DOAN, X. LI, AND V. NGUYEN}{Stable-norm boundary geometry on once-punctured tori}

\section{Introduction}
\label{sec:introduction}

Let $S=S_{1,1}$ be an oriented once-punctured torus. After fixing
$$
H_1(S;\mathbb Z)\cong\mathbb Z^2,
$$
unoriented essential simple closed curves on $S$ are naturally indexed by primitive vectors $w\in\mathbb Z^2$ up to sign, equivalently by $\mathbb P^1(\mathbb Q)$. Thus $w$ and $-w$ determine the same curve. For a complete finite-area hyperbolic structure $X$ on $S$, let $\gamma_w$ be its geodesic representative and $\ell_X(\gamma_w)$ its length.

The norm considered below belongs to the classical stable-norm and
measured-lamination frameworks. Federer \cite{FedererStableNorm} introduced the stable norm on homology,
which was subsequently developed by Gromov
\cite{GromovStableNorm}. Thurston's theory of measured
foliations and laminations provides the corresponding framework for simple
curves on hyperbolic surfaces \cite{FLP}, while Bonahon \cite{Bonahon1988} realized hyperbolic
length as intersection with the Liouville current. For a hyperbolic once-punctured torus $X$, McShane--Rivin
\cite{McShaneRivinNorm} gave the concrete formulation used here. They proved
that the simple-geodesic length function on primitive integral classes extends
to a norm $\|\cdot
\|_X$ on $H_1(S;\mathbb R)$ satisfying
$$
\|w\|_X=\ell_X(\gamma_w)
$$
for every primitive $w\in H_1(S;\mathbb Z)$. Under the induced identification $H_1(S;\mathbb R)\cong \mathbb R^2$, we view
$\|\cdot\|_X$ as a norm on $\mathbb R^2$. By McShane--Rivin, its unit ball
$$
B_X=\{v\in \mathbb R^2:\|v\|_X\leq 1\}
$$
is a centrally symmetric strictly convex body (see Figure \ref{fig:intro-fricke-mr-ball-dilate}). For $L>0$, we write
$$
LB_X=\{v\in\mathbb R^2:\|v\|_X\leq L\},
\qquad
\partial(LB_X)=L\partial B_X=\{v\in\mathbb R^2:\|v\|_X=L\}.
$$
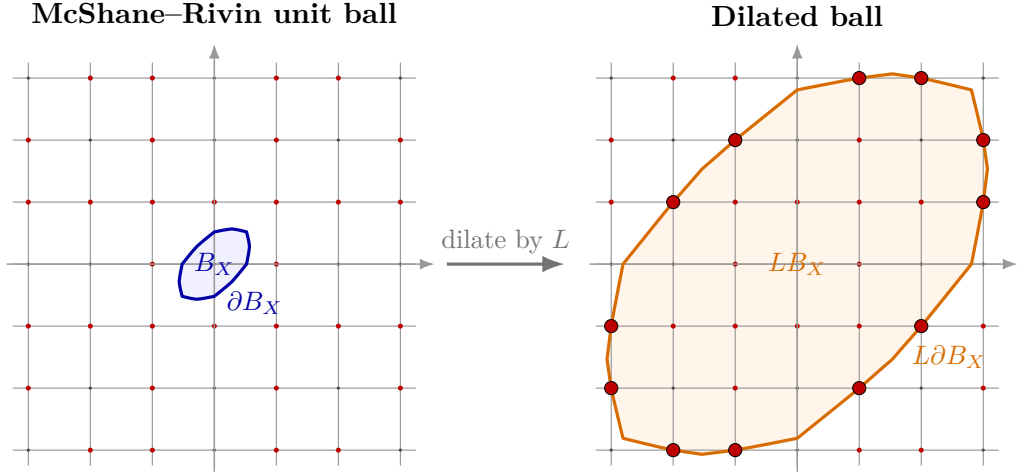
\begin{figure}[h]
\centering
\begin{tikzpicture}[scale=0.82,>=Latex,font=\small]

\def\UnitScale{0.25}
\def\LengthScale{5.408}
\def\Dil{1.352}

\def\MRpoints{
2.078/0.000,
2.148/0.358,
2.162/0.432,
2.184/0.546,
2.219/0.740,
2.239/0.896,
2.269/1.135,
2.239/1.344,
2.219/1.480,
2.184/1.638,
2.162/1.730,
2.148/1.790,
2.078/2.078,
1.790/2.148,
1.730/2.162,
1.638/2.184,
1.480/2.219,
1.344/2.239,
1.135/2.269,
0.896/2.239,
0.740/2.219,
0.546/2.184,
0.432/2.162,
0.358/2.148,
0.000/2.078,
-0.305/1.832,
-0.358/1.790,
-0.432/1.730,
-0.546/1.638,
-0.628/1.571,
-0.740/1.480,
-0.837/1.395,
-0.896/1.344,
-0.963/1.284,
-1.001/1.251,
-1.025/1.230,
-1.135/1.135,
-1.230/1.025,
-1.251/1.001,
-1.284/0.963,
-1.344/0.896,
-1.395/0.837,
-1.480/0.740,
-1.571/0.628,
-1.638/0.546,
-1.730/0.432,
-1.790/0.358,
-1.832/0.305,
-2.078/0.000,
-2.148/-0.358,
-2.162/-0.432,
-2.184/-0.546,
-2.219/-0.740,
-2.239/-0.896,
-2.269/-1.135,
-2.239/-1.344,
-2.219/-1.480,
-2.184/-1.638,
-2.162/-1.730,
-2.148/-1.790,
-2.078/-2.078,
-1.790/-2.148,
-1.730/-2.162,
-1.638/-2.184,
-1.480/-2.219,
-1.344/-2.239,
-1.135/-2.269,
-0.896/-2.239,
-0.740/-2.219,
-0.546/-2.184,
-0.432/-2.162,
-0.358/-2.148,
0.000/-2.078,
0.305/-1.832,
0.358/-1.790,
0.432/-1.730,
0.546/-1.638,
0.628/-1.571,
0.740/-1.480,
0.837/-1.395,
0.896/-1.344,
0.963/-1.284,
1.001/-1.251,
1.025/-1.230,
1.135/-1.135,
1.230/-1.025,
1.251/-1.001,
1.284/-0.963,
1.344/-0.896,
1.395/-0.837,
1.480/-0.740,
1.571/-0.628,
1.638/-0.546,
1.730/-0.432,
1.790/-0.358,
1.832/-0.305
}

\def\MRpath{
(2.078,0.000) --
(2.148,0.358) --
(2.162,0.432) --
(2.184,0.546) --
(2.219,0.740) --
(2.239,0.896) --
(2.269,1.135) --
(2.239,1.344) --
(2.219,1.480) --
(2.184,1.638) --
(2.162,1.730) --
(2.148,1.790) --
(2.078,2.078) --
(1.790,2.148) --
(1.730,2.162) --
(1.638,2.184) --
(1.480,2.219) --
(1.344,2.239) --
(1.135,2.269) --
(0.896,2.239) --
(0.740,2.219) --
(0.546,2.184) --
(0.432,2.162) --
(0.358,2.148) --
(0.000,2.078) --
(-0.305,1.832) --
(-0.358,1.790) --
(-0.432,1.730) --
(-0.546,1.638) --
(-0.628,1.571) --
(-0.740,1.480) --
(-0.837,1.395) --
(-0.896,1.344) --
(-0.963,1.284) --
(-1.001,1.251) --
(-1.025,1.230) --
(-1.135,1.135) --
(-1.230,1.025) --
(-1.251,1.001) --
(-1.284,0.963) --
(-1.344,0.896) --
(-1.395,0.837) --
(-1.480,0.740) --
(-1.571,0.628) --
(-1.638,0.546) --
(-1.730,0.432) --
(-1.790,0.358) --
(-1.832,0.305) --
(-2.078,0.000) --
(-2.148,-0.358) --
(-2.162,-0.432) --
(-2.184,-0.546) --
(-2.219,-0.740) --
(-2.239,-0.896) --
(-2.269,-1.135) --
(-2.239,-1.344) --
(-2.219,-1.480) --
(-2.184,-1.638) --
(-2.162,-1.730) --
(-2.148,-1.790) --
(-2.078,-2.078) --
(-1.790,-2.148) --
(-1.730,-2.162) --
(-1.638,-2.184) --
(-1.480,-2.219) --
(-1.344,-2.239) --
(-1.135,-2.269) --
(-0.896,-2.239) --
(-0.740,-2.219) --
(-0.546,-2.184) --
(-0.432,-2.162) --
(-0.358,-2.148) --
(0.000,-2.078) --
(0.305,-1.832) --
(0.358,-1.790) --
(0.432,-1.730) --
(0.546,-1.638) --
(0.628,-1.571) --
(0.740,-1.480) --
(0.837,-1.395) --
(0.896,-1.344) --
(0.963,-1.284) --
(1.001,-1.251) --
(1.025,-1.230) --
(1.135,-1.135) --
(1.230,-1.025) --
(1.251,-1.001) --
(1.284,-0.963) --
(1.344,-0.896) --
(1.395,-0.837) --
(1.480,-0.740) --
(1.571,-0.628) --
(1.638,-0.546) --
(1.730,-0.432) --
(1.790,-0.358) --
(1.832,-0.305) --
cycle
}

\def\PrimLattice{
-3/-2,-3/-1,-3/1,-3/2,
-2/-3,-2/-1,-2/1,-2/3,
-1/-3,-1/-2,-1/-1,-1/0,-1/1,-1/2,-1/3,
0/-1,0/1,
1/-3,1/-2,1/-1,1/0,1/1,1/2,1/3,
2/-3,2/-1,2/1,2/3,
3/-2,3/-1,3/1,3/2
}

\def\BoundaryPrimitivePoints{
1/-2,
-1/2,
2/-1,
-2/1,
1/3,
-1/-3,
2/3,
-2/-3,
3/1,
-3/-1,
3/2,
-3/-2
}

\tikzset{
  boundary primitive point/.style={
    circle,
    fill=red!75!black,
    draw=black,
    line width=0.25pt,
    inner sep=1.75pt
  }
}

\newcommand{\visiblelattice}{%
  \foreach \x in {-3,-2,-1,0,1,2,3} {
    \draw[black!42,line width=0.42pt] (\x,-3.25) -- (\x,3.25);
  }
  \foreach \y in {-3,-2,-1,0,1,2,3} {
    \draw[black!42,line width=0.42pt] (-3.25,\y) -- (3.25,\y);
  }
  \foreach \x in {-3,-2,-1,0,1,2,3} {
    \foreach \y in {-3,-2,-1,0,1,2,3} {
      \fill[black!70] (\x,\y) circle (0.75pt);
    }
  }
  \foreach \x/\y in \PrimLattice {
    \fill[red!75!black] (\x,\y) circle (1.15pt);
  }
}

\begin{scope}[xshift=-4.7cm]

\node[font=\bfseries] at (0,4.05) {Unit ball};

\begin{scope}[scale=\UnitScale]
\path[
  fill=black!6,
  draw=none,
  line join=round
] \MRpath;
\end{scope}

\visiblelattice

\draw[->,gray!80,line width=0.55pt] (-3.35,0) -- (3.55,0);
\draw[->,gray!80,line width=0.55pt] (0,-3.35) -- (0,3.55);

\begin{scope}[scale=\UnitScale]
\path[
  draw=black!65!black,
  very thick,
  fill=none,
  line join=round
] \MRpath;
\end{scope}

\foreach \x/\y in \MRpoints {
  \fill[black!65!black] ({\UnitScale*\x},{\UnitScale*\y}) circle (0.38pt);
}

\node[black!65!black] at (0,0) {$B_X$};
\node[black!65!black] at (0.65,-0.62) {$\partial B_X$};

\end{scope}

\begin{scope}[xshift=4.7cm]

\node[font=\bfseries] at (0,4) {Dilated ball};

\begin{scope}[scale=\Dil]
\path[
  fill=black!8,
  draw=none,
  line join=round
] \MRpath;
\end{scope}

\visiblelattice

\draw[->,gray!80,line width=0.55pt] (-3.35,0) -- (3.55,0);
\draw[->,gray!80,line width=0.55pt] (0,-3.35) -- (0,3.55);

\begin{scope}[scale=\Dil]
\path[
  draw=black!85!black,
  very thick,
  fill=none,
  line join=round
] \MRpath;
\end{scope}

\foreach \x/\y in \MRpoints {
  \fill[black!85!black] ({\Dil*\x},{\Dil*\y}) circle (0.35pt);
}

\foreach \x/\y in \BoundaryPrimitivePoints {
  \node[boundary primitive point] at (\x,\y) {};
}

\node[black!85!black] at (2.45,-1.5) {$L\partial B_X$};

\node[black!85!black] at (0,0) {$L B_X$};


\end{scope}

\draw[->,very thick,black!55] (-0.95,0) -- (0.95,0)
  node[midway,above] {dilate by $L$};

\end{tikzpicture}
\caption{The unit ball $B_X$, where $X$ is the modular torus. The three shortest unoriented primitive classes, represented by $(1,0)$, $(0,1)$, and $(1,1)$, have trace $3$. The dilate $LB_X$, where $L=2\arccosh(15/2)$, has 12 red primitive lattice points on $L\partial B_X$ correspond to $6$ unoriented simple closed geodesics of length $L$ and trace $15$.}

\label{fig:intro-fricke-mr-ball-dilate}
\end{figure}

\subsection*{Local boundary geometry.}
The local shape of $\partial B_X$ carries arithmetic information. In the modular case, Aigner formulated monotonicity conjectures for the Markoff ordering \cite{Aigner}, later proved and refined by continued-fraction and cluster-algebra methods \cite{RabideauSchiffler2020,LLRS,LPVT}. Shortly thereafter, McShane gave a unified geometric proof by applying  the convexity of the
norm ball \cite{McShaneAigner}.
Building on this work, Gaster \cite{GasterBoundary} computed the one-sided slopes at the rational
corners of the modular norm ball and used them to give a sharp slope interval for which the Markoff ordering is monotone in the denominator,
confirming the refined conjectures of Lee--Li--Rabideau--Schiffler
\cite{LLRS}. This is the modular model for the normal-turn estimate in this paper.

More generally, for every once-punctured torus $X$, McShane and Rivin stated that rational directions give genuine corners of $\partial B_X$ with exponentially small exterior angles, whereas the boundary points
lying over irrational directions are flat to infinite order \cite{McShaneRivinTori}. Their paper gives a sketch of the
argument, but, as noted in \cite{GasterBoundary,SorrentinoVeselov}, a complete
proof of this full local statement does not seem to have appeared.
The local regularity of stable norms is closely related to Mather's
$\beta$-function and to Bangert's work on minimal geodesics and stable norms
\cite{Mather1991,Bangert1988,Bangert1990,bangert1994geodesic}.
Sorrentino--Veselov \cite{SorrentinoVeselov} connected this framework to the
once-punctured torus and, using Bangert's theorem, proved differentiability at
irrational directions and non-differentiability at rational directions. They
also asked for more quantitative information about the rational corners. This leads to a finer
local problem: to estimate the corner sizes and the flatness orders at irrational
directions. 

Our local-geometry result (Theorem~\ref{thm:intro-local-regularity}) addresses
this problem and corrects the irrational part of the McShane--Rivin picture:
boundary points lying over irrational directions are not uniformly flat to
infinite order. Instead, the boundary acts as a
convex-geometric detector of exponential rational approximation. In the usual slope coordinate, the relevant arithmetic scale is
$$
\limsup_{j\to\infty}\frac{\log q_{j+1}}{q_j},
$$
where $q_j$ are the denominators of the continued-fraction convergents. We show that
subexponential rational approximation gives infinite flatness, finite
exponential approximation gives finite flatness order, and approximation
faster than every fixed exponential scale gives the least possible irrational
flatness. This is the same exponential-Diophantine threshold that appears in
continued fractions, cusp excursions, homogeneous dynamics, and small-divisor
problems \cite{SeriesModularSurface,SullivanLogLaw,Dani1985,KleinbockMargulis1998,AvilaARC,AvilaGlobal}.

\subsection*{Simple-Length Counting.}
The dilated balls $LB_X$ naturally encode cumulative simple-length counting. Indeed, define
$$
N_X(L)\overset{\mathrm{def}}{=}
\#\{\gamma:\gamma\text{ is an unoriented simple closed geodesic and }
\ell_X(\gamma)\leq L\},
$$
then estimating $N_X(L)$ is exactly the problem of counting primitive lattice
points in $LB_X$, modulo sign. From this point of view, McShane and Rivin proved the (best-known)
estimate \cite{McShaneRivinNorm,McShaneRivinTori}:
\begin{equation}\label{eq:MRestimate}
    N_X(L)=\frac{3\operatorname{Area}(B_X)}{\pi^2}L^2+O_X(L\log L).
\end{equation}
For general finite-type hyperbolic surfaces $X$, Rivin \cite{Rivin2001} proved the correct polynomial order of growth for $N_X(L)$ by a topological counting
argument based on arc systems and twist parameters. Mirzakhani later proved the precise asymptotic formula for $N_X(L)$ using measured laminations, mapping class group
dynamics, and Weil--Petersson volume recursions arising from generalized
McShane identities \cite{mcshane1998simple,Mirzakhani}. Related counting results and the geodesic-current approach were developed by Erlandsson--Souto \cite{ErlandssonSoutoCounting,ErlandssonSoutoBook}.

Beyond cumulative counting, it is natural to ask for the corresponding
boundary-counting problem: how many lattice points lie on $L\partial B_X$ rather
than in $LB_X$ (see Figure \ref{fig:intro-fricke-mr-ball-dilate}). A classical theorem of Jarník \cite{Jarnik} gives the universal bound
$$
\#(\gamma\cap\mathbb Z^2)\leq C\,\Len(\gamma)^{2/3}
$$
for every closed strictly convex plane curve $\gamma$. In particular, for every
fixed such curve $\gamma$, there exists $C_\gamma>0$ such that, for all
$L\geq 1$,
$$
\#(L\gamma\cap\mathbb Z^2)\leq C_\gamma L^{2/3}.
$$
The exponent $2/3$ is optimal for general strictly convex curves. Gaster asked whether the special curve $\partial B_X$ for the modular torus satisfies a stronger estimate \cite{GasterBoundary,BIRSMarkov2025}. Our boundary-counting result answers this question in greater generality: for every hyperbolic once-punctured torus $X$, there is a constant $C_X>0$ such that, for all $L\ge 2$,
$$
\#(L\partial B_X\cap\mathbb Z^2)\leq C_X(\log L)^2.
$$

On the other hand, after restricting to primitive lattice points and quotienting by sign, this boundary count is exactly the simple-length multiplicity (see Figure \ref{fig:intro-fricke-mr-ball-dilate}):
$$
m_X(L)
\overset{\mathrm{def}}{=}
\#\{\gamma:\gamma\text{ is an unoriented simple closed geodesic and }
\ell_X(\gamma)=L\}.
$$

Unlike cumulative counting, simple-length multiplicity on a fixed surface is a finer and substantially harder problem. Schmutz Schaller \cite{SchmutzSchaller} conjectured the sharp uniform bound
$$m_X(L)\le 6,$$ 
for every hyperbolic once-punctured torus $X$ and every length $L>0$. On the modular torus, simple closed geodesics have absolute traces $3m$ and lengths $2\operatorname{arcosh}(3m/2)$, with $m$ a Markoff number; this correspondence was found independently by Gorshkov and Cohn \cite{Gorshkov1981,Cohn1955,Cohn1971,Cohn1972}. Thus Schmutz Schaller's bound is a hyperbolic-geometric extension of the Frobenius unicity conjecture for Markoff numbers \cite{Frobenius1913,Aigner}. McShane and Parlier \cite{McShaneParlier} proved that the simple length spectrum is
Baire-generically simple on the Teichmüller space and showed that a naive one-holed torus analog of the Schmutz Schaller bound fails: along the
symmetric line, a dense set of tori has at least $12$ distinct simple closed
geodesics of the same length. Using stable norms, Makover--Parlier--Sutton proved a position theorem for equal simple lengths on tori \cite{makover2012constructing}, which, combined with the McShane--Rivin cumulative estimate, recovers the Jarn\'ik-scale bound $$m_X(L)\leq C_XL^{2/3}$$ for  once-punctured tori $X$. This remains far from Schmutz Schaller's conjectural
uniform bound $m_X(L)\leq 6$. Our result improves the general bound to
$$
m_X(L)\leq C_X(\log L)^2.
$$
For general finite-type hyperbolic surfaces, it remains open whether
multiplicities in the simple length spectrum are bounded above by a constant
depending only on the topology of the surface.
This question goes back to Schmutz Schaller's survey \cite{schaller1998geometry}, where the conjecture is attributed to Rivin,
see also \cite{McShaneParlier,parlier2014simple}.
A related bounded-multiplicity
problem for the simple Teichm\"uller length spectrum of moduli space appears
in Farb's problem list \cite[Question~7.7]{farb2006problems}.

\section*{Result statements}
The common geometric input is an exponential decay of the boundary's normal
turn, that is, the total variation of its outward normal directions. For a
rational direction represented by a primitive vector $w=(p,q)$, define its
height by
$$
\height(w)=\max\{|p|,|q|\}.
$$
Write $\Turn(E)$ for the total normal turn along a boundary set
$E\subset\partial B_X$. We prove that, after removing all rational boundary
points of height at most $H$, the remaining normal turn is exponentially
small.

\begin{theorem}
\label{thm:tail-turn}
There are constants $C_X,\kappa_X>0$ such that, for every $H\geq1$,
$$
\Turn\left(
\partial B_X\setminus
\left\{
\pm\frac{w}{\|w\|_X}:
w\in\mathbb Z^2\text{ primitive},\ \height(w)\leq H
\right\}
\right)
\leq C_Xe^{-\kappa_XH}.
$$
\end{theorem}

We first use this estimate to give a sharp local description of $\partial B_X$, refining and correcting the
McShane--Rivin local picture \cite{McShaneRivinTori}.

\begin{theorem}
\label{thm:intro-local-regularity}
The boundary $\partial B_X$ has the following local behavior.

\begin{itemize}
\item Rational directions give genuine corners. More precisely, there are
constants $c_X,C_X>0$ such that, if $w$ is a primitive representative of a
rational direction, then its corner angle $\alpha_X(w)$ satisfies
$$
c_X\height(w)e^{-\|w\|_X}
\leq
\alpha_X(w)
\leq
C_X\height(w)e^{-\|w\|_X}.
$$

\item For every irrational direction $\beta$, the boundary $\partial B_X$ has
a unique supporting line at each of the two antipodal points of
$\beta\cap\partial B_X$.

\item For an irrational direction $\beta$, define
$$
\Omega(\beta)
=
\limsup_{\bar w\to\beta}
\frac{\log(1/d_{\mathbb P^1}(\bar w,\beta))}{\height(\bar w)},
$$
where $\bar w$ ranges over rational directions. Let $\|p_\beta\|_\infty$ denote the max norm of either point
$p_\beta\in\beta\cap\partial B_X$. Then the graph-flatness order
(see Definition \ref{def:graph-flatness}) of $\partial B_X$ at either point of
$\beta\cap\partial B_X$ is
$$
\mathfrak f_X(\beta)
=
1+\frac{1}{\Omega(\beta)\|p_\beta\|_\infty},
\qquad 0\le\Omega(\beta)\le\infty,
$$
with the usual conventions at $0$ and $\infty$.

\item Set
$$
\mathcal I=\mathbb P^1(\mathbb R)\setminus\mathbb P^1(\mathbb Q),
\qquad
\mathcal F_A=\mathcal F_A(X)=\{\beta\in\mathcal I:\mathfrak f_X(\beta)=A\}.
$$
Then the flatness strata have the following size properties.

\begin{itemize}
\item Every flatness order occurs densely: for every $A\in[1,\infty]$, every
nonempty interval in $\mathcal I$ contains some direction in $\mathcal F_A$.

\item The least-flatness set $\mathcal F_1$ is topologically large but
metrically small: it is a dense $G_\delta$ subset of $\mathcal I$, but has
Hausdorff dimension $0$.

\item The infinite-flatness set $\mathcal F_\infty$ is metrically large but
topologically small: it has full Lebesgue measure, its complement has
Hausdorff dimension $0$, but $\mathcal F_\infty$ is meagre in $\mathcal I$.

\item Every intermediate finite level is small in both senses: for every
$1<A<\infty$, the set $\mathcal F_A$ is meagre in $\mathcal I$ and has
Hausdorff dimension $0$.
\end{itemize}

\end{itemize}
\end{theorem}
The proof is based on the normal turn of the boundary. Rational directions
produce corner atoms of size comparable to
$
\height(w)e^{-\|w\|_X}.
$
Near an irrational direction $\beta$, the question is whether these rational
corners are close enough to remain visible after zooming in. Thus the flatness
order is controlled by a balance between two quantities: the size of a rational
corner, and the projective distance from that rational direction to $\beta$.

The intrinsic exponent $\Omega(\beta)$ measures this balance. We show that it has a simple continued-fraction
interpretation in slope coordinates (see Lemma \ref{lem:continued-fractions-slope}). If
$\beta=[(1,\mathfrak b)]$ and $p_j/q_j$ are the continued-fraction convergents
of $\mathfrak b$, then
$$
\Omega(\beta)
=
\frac{1}{\max\{1,|\mathfrak b|\}}
\limsup_{j\to\infty}\frac{\log q_{j+1}}{q_j}.
$$
Thus the flatness order is governed by the largest exponential jumps in the
continued-fraction denominators of the slope. By prescribing these jumps, one
obtains irrational directions with any desired flatness order. The Hausdorff
dimension and category statements then follow from standard covering and
Baire-category arguments.

The borderline case
$\mathfrak f_X(\beta)=1$ occurs when rational directions approach $\beta$
faster than every fixed exponential rate relative to their height. Their corner
angles are small, but their projective distances to $\beta$ can be much
smaller. After rescaling near $p\in\beta\cap\partial B_X$, these nearby corners
remain visible and can make the boundary look almost V-shaped at arbitrarily
fine scales; see the middle panel of
Figure~\ref{fig:corner-like-vs-rational-corners} for an illustration. At the opposite extreme, when $\Omega(\beta)=0$, no
rational directions approach $\beta$ at any fixed exponential rate, and the
boundary is flat to infinite order at $p$; see the rightmost panel of
Figure~\ref{fig:corner-like-vs-rational-corners} for an illustration.
We remark that Figure~\ref{fig:corner-like-vs-rational-corners} is schematic:
near an irrational direction, the two sides of the boundary need not look
symmetric, since rational directions may approach at different exponential
rates from the left and from the right. The graph-flatness order is two-sided,
so it is determined by the less flat side.

\begin{figure}[!htbp]
\centering
\begin{tikzpicture}[scale=1.0,>=Latex,font=\small]

\tikzset{
  boundary/.style={black!70!black,very thick,line cap=round,line join=round},
  support/.style={gray!70,dashed,line width=0.55pt},
  corner/.style={blue!85!black,thick},
  point/.style={black,fill=black,circle,inner sep=1.1pt},
  body/.style={black!6},
  panel title/.style={font=\small\bfseries,align=center}
}

\begin{scope}[xshift=-3.6cm]

\node[panel title] at (0,1.40) {Rational: genuine corner};

\fill[body]
(-1.25,0.95) -- (1.25,0.95) -- (1.25,0.42) --
(0,0) -- (-1.25,0.42) -- cycle;


\draw[boundary] (-1.15,0.38) -- (0,0) -- (1.15,0.38);

\node[point] at (0,0) {};




\end{scope}

\begin{scope}[xshift=1.5cm]

\node[panel title] at (0,1.40) {Irrational: $\mathfrak f_X(\beta)=1$};

\fill[body]
(-1.25,0.95) -- (1.25,0.95) --
(1.25,0.72) --
(0.95,0.38) --
(0.45,0.075) --
(0.15,0.005) --
(0.00,0.000) --
(-0.15,0.005) --
(-0.45,0.075) --
(-0.95,0.38) --
(-1.25,0.72) --
cycle;

\draw[support] (-1.25,0) -- (1.25,0);

\draw[boundary,smooth]
plot coordinates {
(-1.25,0.72)
(-0.95,0.38)
(-0.45,0.075)
(-0.15,0.005)
(0.00,0.000)
(0.15,0.005)
(0.45,0.075)
};
\draw[boundary] (0.45,0.075) -- (0.95,0.38) -- (1.20,0.68);

\node[point] at (0,0) {};






\end{scope}

\begin{scope}[xshift=7cm]

\node[panel title] at (0,1.40) {Irrational: $\mathfrak f_X(\beta)=\infty$};

\fill[body]
(-1.25,0.95) -- (1.25,0.95) --
(1.25,0.25) --
(1.10,0.170) --
(0.75,0.035) --
(0.35,0.002) --
(0,0) --
(-0.35,0.002) --
(-0.75,0.035) --
(-1.10,0.170) --
(-1.25,0.25) --
cycle;

\draw[support] (-1.25,0) -- (1.25,0);

\draw[boundary,smooth]
plot coordinates {
(-1.10,0.170)
(-0.75,0.035)
(-0.35,0.002)
(0,0)
(0.35,0.002)
(0.75,0.035)
(1.10,0.170)
};

\node[point] at (0,0) {};



\end{scope}

\end{tikzpicture}

\caption{A comparison of local boundary behavior.}
\label{fig:corner-like-vs-rational-corners}
\end{figure}

The rational-corner computation also leads to an new identity for simple closed geodesics. Under
duality, each rational corner of $B_X$ corresponds to an edge of the dual body
$B_X^*$, and the exact corner formula determines the area contributed by this
edge. This motivates considering the dual body (see
Figure~\ref{fig:dual-introduction})
$$
B_X^*
=
\{\lambda\in\mathbb R^2:\langle\lambda,v\rangle\leq1
\text{ for every }v\in B_X\}.
$$
\begin{figure}[ht]
\centering
\begin{tikzpicture}[x=1.25cm,y=1.25cm,line cap=round,line join=round]
  \draw[->,black!35] (-2.25,0)--(2.25,0);
  \draw[->,black!35] (0,-2.25)--(0,2.25);

  \draw[thick,dashed] plot coordinates {
    (-1.600647,-0.324200) (-0.324200,-1.600647)
    (0.324200,-1.924847) (1.557457,-1.924847)
    (1.643837,-1.881657) (1.881657,-1.643837)
    (1.924847,-1.557457) (1.924847,-0.324200)
    (1.600647,0.324200) (0.324200,1.600647)
    (-0.324200,1.924847) (-1.557457,1.924847)
    (-1.643837,1.881657) (-1.881657,1.643837)
    (-1.924847,1.557457) (-1.924847,0.324200)
  } -- cycle;

  \fill[black!10] plot coordinates {
    (-0.567296,-0.283648) (-0.519522,-0.519522)
    (-0.283648,-0.567296) (-0.000000,-0.519522)
    (0.184940,-0.369881) (0.283648,-0.283648)
    (0.369881,-0.184940) (0.519522,0.000000)
    (0.567296,0.283648) (0.519522,0.519522)
    (0.283648,0.567296) (0.000000,0.519522)
    (-0.184940,0.369881) (-0.283648,0.283648)
    (-0.369881,0.184940) (-0.519522,-0.000000)
  } -- cycle;
  \draw[very thick] plot coordinates {
    (-0.567296,-0.283648) (-0.519522,-0.519522)
    (-0.283648,-0.567296) (-0.000000,-0.519522)
    (0.184940,-0.369881) (0.283648,-0.283648)
    (0.369881,-0.184940) (0.519522,0.000000)
    (0.567296,0.283648) (0.519522,0.519522)
    (0.283648,0.567296) (0.000000,0.519522)
    (-0.184940,0.369881) (-0.283648,0.283648)
    (-0.369881,0.184940) (-0.519522,-0.000000)
  } -- cycle;

  \draw[->,very thick] (0,0)--(1.924847,-0.940829)
    node[pos=1,right] {$\lambda$};
  \fill (1.924847,-0.940829) circle (1.25pt);

  \draw[dotted,thick] (-0.555797,-2.200000)--(1.594840,2.200000);
  \draw[dotted,thick] (-1.594840,-2.200000)--(0.555797,2.200000);

  \fill (0.519522,0) circle (1.2pt) node[right=1pt] {$v$};
  \fill (-0.519522,0) circle (1.2pt) node[left=1pt] {$-v$};

  \node[rotate=64,anchor=south] at (1.37,1.64)
    {$\langle \lambda,\cdot\rangle=1$};
  \node[rotate=64,anchor=north] at (-1.39,-1.66)
    {$\langle \lambda,\cdot\rangle=-1$};

  \node at (-0.0,0.20) {$B_X$};
  \node at (-1.50,1.45) {$B_X^*$};
\end{tikzpicture}
\caption{The unit ball $B_X$ and its dual body $B_X^*$,  where $X$ is the modular torus.}
\label{fig:dual-introduction}
\end{figure}
Summing these contributions over all rational corners gives the following
McShane-type identity.
\begin{theorem}
\label{thm:dual-area}
For every complete finite-area hyperbolic once-punctured torus $X$,
$$
\operatorname{Area}(B_X^*)
=
4\sum_{\gamma}
\ell_X(\gamma)
\log\coth\left(\frac{\ell_X(\gamma)}{2}\right),
$$
where the sum is over unoriented simple closed geodesics on $X$.
\end{theorem}

Geometrically, each edge and the origin span a triangle of area
$$
2\ell_X(\gamma_w)
\log\coth\left(\frac{\ell_X(\gamma_w)}{2}\right),
$$
and the antipodal edge contributes the same amount (see Figure \ref{fig:dual-triangulation}).

\begin{figure}[ht]
\centering
\begin{tikzpicture}[x=1.25cm,y=1.25cm,line cap=round,line join=round]
  \coordinate (O) at (0,0);
  \coordinate (Lm) at (1.924847,-1.557457);
  \coordinate (Lp) at (1.924847,-0.324200);
  \coordinate (mLm) at (-1.924847,1.557457);
  \coordinate (mLp) at (-1.924847,0.324200);

  \fill[black!4] plot coordinates {
    (-1.600647,-0.324200) (-0.324200,-1.600647)
    (0.324200,-1.924847) (1.557457,-1.924847)
    (1.643837,-1.881657) (1.881657,-1.643837)
    (1.924847,-1.557457) (1.924847,-0.324200)
    (1.600647,0.324200) (0.324200,1.600647)
    (-0.324200,1.924847) (-1.557457,1.924847)
    (-1.643837,1.881657) (-1.881657,1.643837)
    (-1.924847,1.557457) (-1.924847,0.324200)
  } -- cycle;

  \foreach \x/\y in {
    -1.600647/-0.324200,-0.324200/-1.600647,
    0.324200/-1.924847,1.557457/-1.924847,
    1.643837/-1.881657,1.881657/-1.643837,
    1.924847/-1.557457,1.924847/-0.324200,
    1.600647/0.324200,0.324200/1.600647,
    -0.324200/1.924847,-1.557457/1.924847,
    -1.643837/1.881657,-1.881657/1.643837,
    -1.924847/1.557457,-1.924847/0.324200
  }{
    \draw[black!28] (O)--(\x,\y);
  }

  \fill[black!20] (O)--(Lm)--(Lp)--cycle;
  \fill[black!20] (O)--(mLm)--(mLp)--cycle;

  \draw[very thick] plot coordinates {
    (-1.600647,-0.324200) (-0.324200,-1.600647)
    (0.324200,-1.924847) (1.557457,-1.924847)
    (1.643837,-1.881657) (1.881657,-1.643837)
    (1.924847,-1.557457) (1.924847,-0.324200)
    (1.600647,0.324200) (0.324200,1.600647)
    (-0.324200,1.924847) (-1.557457,1.924847)
    (-1.643837,1.881657) (-1.881657,1.643837)
    (-1.924847,1.557457) (-1.924847,0.324200)
  } -- cycle;

  \draw[ultra thick] (Lm)--(Lp);
  \draw[ultra thick] (mLm)--(mLp);

  \fill (O) circle (1.25pt) node[below left] {$0$};

  \node at (1,1.48) {$B_X^*$};
\end{tikzpicture}
\caption{Triangles from the origin over the rational edges of $B_X^*$.}
\label{fig:dual-triangulation}
\end{figure}


This convex-geometric viewpoint suggests a systematic framework for studying many further identities for simple closed geodesics, which we leave for future work in order to keep the present paper focused. A second global application of Theorem~\ref{thm:tail-turn} concerns the boundary-counting problem and gives an answer to the question of Gaster \cite{GasterBoundary,BIRSMarkov2025}, improving the Jarn\'ik-scale bound of order $L^{2/3}$ obtained from classical convex lattice-point estimates \cite{Jarnik} and from the work of Makover--Parlier--Sutton \cite{makover2012constructing}.
\begin{theorem}
\label{thm:main}
For every complete finite-area hyperbolic once-punctured torus $X$, there is a
constant $C_X>0$ such that for every $L\geq 2$,
$$
\#(L\partial B_X\cap\Z^2)\leq C_X(\log L)^2,
$$
and consequently, $m_X(L) \le C_X(\log L)^2$.
\end{theorem}
The key point is that an arc with small normal turn has chord directions
confined to a narrow angular sector. A localized Jarn\'ik estimate then gives
a sharper lattice-point bound than the global $L^{2/3}$ estimate. Taking
$H$ of order $\log L$ yields the stated $(\log L)^2$ bound.
\subsection*{Markoff-fiber estimates}
On the modular torus, simple closed geodesics have absolute traces $3m$ and lengths $L_m=2\operatorname{arcosh}(3m/2)$, with $m$ a Markoff number. We normalize the Markoff fiber $\#\lambda_M^{-1}(m)$ as the number of primitive homology classes modulo sign, equivalently the number of unoriented simple closed geodesics of length $L_m$. This is the all-slope convention; Gaster's sector convention \cite{GasterBoundary} differs from it by at most the six images of a standard sector under the finite symmetry group of $M$. Applying
Theorem \ref{thm:main} with $L=L_m$ gives:

\begin{corollary}
\label{cor:markoff-fibers}
There is a constant $C>0$ such that for every Markoff number $m$,
$$
\#\lambda_M^{-1}(m)\leq C(\log\log(3m))^2.
$$
\end{corollary}

For context, the conjectural benchmark for Corollary \ref{cor:markoff-fibers} is the long-standing Frobenius unicity conjecture for Markoff numbers \cite{Frobenius1913,Aigner}.
 It concerns positive
integer solutions $(a,b,c)$ of the Markoff equation
$$
a^2+b^2+c^2=3abc.
$$
The conjecture asserts that a Markoff triple $(a,b,c)$ with $a\leq b\leq c$ is
determined by its largest entry $c$. By the modular-torus correspondence
recalled above, this is equivalent to the expected uniform bound
$
\#\lambda_M^{-1}(m)\leq 6
$
in our geometric convention. Many arithmetic partial results toward the Markoff uniqueness conjecture are known: Button \cite{button2001markoff} proved uniqueness for largest Markoff numbers $m=kp^\ell$ with $1\leq k\leq 10^{35}$, while Chen--Chen \cite{ChenChen2013} proved it when either $3m-2$ or $3m+2$ equals $kp^\ell$ with $1\leq k\leq 10^{10}$, where $p$ is prime and $\ell\geq 1$; see also \cite{Baragar1996,Button1998,Schmutz1996,LangTan,Zhang2007}. Button's family is
arithmetically sparse in the Markoff multiset $M^s$: Bourgain--Gamburd--Sarnak
proved that, for every fixed $\nu\geq 1$, almost every element of $M^s$ has more
than $\nu$ distinct prime factors \cite{bourgain2026strong}.
 For a recent overview of progress, open problems, and possible directions, see the BIRS workshop report \cite{BIRSMarkov2025}.

Our result gives a growing upper bound, still far from the conjectural uniform
bound. Nevertheless, it improves the previously known general estimates.
As noted by Gaster \cite[Remark 1.13]{GasterBoundary}, Zagier's
tree analysis \cite{Zagier1982} gives the logarithmic fiber bound
$$
\#\lambda_M^{-1}(m)=O(\log m).
$$
Another logarithmic bound follows from Lee--Li--Rabideau--Schiffler
\cite[Corollary 1.5(b)]{LLRS}. The error term in the McShane--Rivin
cumulative counting in Equation \eqref{eq:MRestimate} for the modular torus gives a weaker estimate
$$
\#\lambda_M^{-1}(m)=O(\log m\log\log m).
$$
Our Corollary \ref{cor:markoff-fibers} gives a much smaller estimate. In particular, the Markoff-fiber size is subpolynomial in $\log m$: for every
$\epsilon>0$,
$$
\#\lambda_M^{-1}(m)=O_{\epsilon}((\log m)^\epsilon).
$$
\section*{Organization of the paper}
Section~\ref{sec:preliminaries} collects the preliminary material
used throughout the paper: the stable norm, the height convention, the
Farey-tree notation, the Fricke trace relation, Bowditch's sink theorem, the
height-gap decomposition, and the convex-geometric language of supporting
functionals and normal turn. Section~\ref{sec:one-gap-turn} proves the
exponential tail estimate for the normal turn away from low-height rational
directions. Section~\ref{sec:local-regularity} proves the rational-corner
estimates and the dual-area identity, then uses the same turn estimate locally
to prove the exact irrational flatness formula and its rigidity consequence.
Section~\ref{sec:lattice-count} combines the tail estimate with a localized
Jarník-type lattice count and proves the global counting theorem. The last
section discusses further directions and generalizations.

\section*{Acknowledgements}

The first-named author thanks Ser Peow Tan, Dídac Martínez-Granado, and Quang-Khai Nguyen for fruitful discussions. The later parts of Theorem \ref{thm:intro-local-regularity} grew especially from discussions with Quang-Khai Nguyen and observations suggested by Ser Peow Tan during the first-named author’s visit to the National University of Singapore (NUS). He is extremely grateful to the Department of Mathematics at NUS for its hospitality, support, and stimulating research environment during his visits and 2-year postdoctoral stay. He thanks the organizers of the BIRS workshop \emph{Perspectives on Markov Numbers} for their invitation and support, despite his inability to attend due to visa issues. He also acknowledges support from the China Exchange Program (JC202512007F) for his visit to the second-named author at Southwest Jiaotong University and thanks its Department of Mathematics for their warm hospitality. His research was partially supported by the Singapore National Research Foundation (NRF) under grant E-146-00-0029-01.

The first two authors also thank the organizers of
\emph{Quantum Topology and Hyperbolic Geometry} in Phu Quoc, Vietnam, and
\emph{Teichm\"uller Theory and Related Topics} in Hefei, China, where they first met and discussed this project.

\section{Preliminaries}
\label{sec:preliminaries}

This section fixes the notation and standard inputs used throughout the paper.
We first set up primitive homology classes, height, and the stable
norm. We then recall the Farey-tree model, the Fricke trace relation,
Bowditch's sink theorem, and the height-gap decomposition. Finally, we introduce
the convex-geometric notation for supporting functionals, one-sided support
directions, and normal turn of the boundary of the unit ball.

\textbf{Convention.} 
Let $X$ be a complete finite-area hyperbolic once-punctured torus. Throughout the paper, we fix an integral basis of $H_1(S;\mathbb Z)$ and use it
to identify
$$
H_1(S;\mathbb Z)\cong \mathbb Z^2,
\qquad
H_1(S;\mathbb R)\cong \mathbb R^2.
$$

An (irrational or rational) direction means an unoriented
line through the origin in $H_1(S;\mathbb R)$, equivalently, an element of
$\mathbb P^1(\mathbb R)$. A direction is rational if it lies in $\mathbb P^1(\mathbb Q)$,
and irrational otherwise.

Constants denoted by $C_X,c_X,\kappa_X$, and similar symbols may change from
line to line and depend only on $X$, after the above background choices have
been fixed. Constants denoted by $C$ are absolute.

\subsection{Primitive classes, height, and the stable norm}

For $w=(p,q)\in\mathbb Z^2$, define
$$
\height(w) \overset{\mathrm{def}}{=} \max\{|p|,|q|\}.
$$
Thus $\height(\bar w)$ is well-defined for $\bar w\in\Prim/\{\pm 1\}=\mathbb P^1(\mathbb Q)$, since $w$ and $-w$ have the same height.

Primitive homology classes modulo sign correspond to unoriented simple closed
curves, up to free homotopy, on the once-punctured torus. We denote by $\gamma_w$ the closed geodesic
associated to the pair $\{w,-w\}$.

We need the following theorem of McShane--Rivin
\cite{McShaneRivinTori,McShaneRivinNorm}. It says that the hyperbolic length
of a simple closed geodesic is the value of a genuine norm on homology.

\begin{inputthm}
\label{thm:mr}
There is a norm
$$
\|\cdot\|_X:H_1(S;\R)\to\R_{\geq 0}
$$
such that, for every $w\in\Prim$,
$$
\|w\|_X=\ell_X(\gamma_w).
$$
Its unit ball
$$
B_X  \overset{\mathrm{def}}{=} \{v\in\R^2:\|v\|_X\leq 1\}
$$
is compact, centrally symmetric, and strictly convex. Equivalently,
$$
\partial B_X \overset{\mathrm{def}}{=}\{v\in\R^2:\|v\|_X=1\}
$$
is a compact convex curve with no nontrivial line segment.
\end{inputthm}

\begin{remark}
Strict convexity is essential for boundary lattice counting: a nontrivial
boundary segment with rational endpoints would give suitable dilates containing
linearly many primitive lattice points.
\end{remark}

We shall repeatedly use the following comparison between height and the stable
norm. Since
$$
\height(v)=\|v\|_\infty=\max\{|v_1|,|v_2|\}
$$
on $\mathbb Z^2$, and all norms on $\mathbb R^2$ are comparable, there are
constants $c_X,C_X>0$ such that
\begin{equation}\label{eq:NormComparison}
    c_X\height(v)\leq \|v\|_X\leq C_X\height(v)
\end{equation}
for every $v\in\mathbb Z^2$.

\subsection{Trace labels and the Farey tree}

For $w\in\Prim$, set
$$
x_w  \overset{\mathrm{def}}{=} 2\cosh\frac{\|w\|_X}{2}.
$$
Since $x_w=x_{-w}$, this gives a well-defined trace label $x_{\bar w}$ on each
projective primitive class $\bar w\in\Prim/\{\pm 1\}$.

We use the Farey graph whose vertices are the elements of $\Prim/\{\pm 1\}$. Two
vertices are joined when they have representatives $u,v\in\Prim$ with
$
\det(u,v)=1.
$

Let $\mathcal T$ be the trivalent tree dual to the Farey graph. Its
complementary regions are naturally indexed by $\Prim/\{\pm 1\}$. We write
$R_{\bar w}$ for the complementary region labelled by $\bar w$; after choosing
a representative $w\in\Prim$, we also write
$$
R_w=R_{\bar w}.
$$
Thus $R_w=R_{-w}$. The region $R_w$ carries the trace label $x_w$.

A vertex of $\mathcal T$ is incident to three complementary regions. These
three region labels form a Farey triangle. An edge $e$ of $\mathcal T$ joins
two vertices whose Farey triangles share two labels. More concretely, if the
shared labels have representatives $u,v\in\Prim$ with
$
\det(u,v)=1,
$
then the two endpoints of $e$ are incident to
$
\{R_u,R_v,R_{u+v}\}
$
and
$
\{R_u,R_v,R_{u-v}\},
$
respectively. We call $u+v$ and $u-v$ the two third labels across $e$; their
trace labels are $x_{u+v}$ and $x_{u-v}$.

The following combines the Fricke trace identity with Bowditch's sink theorem
for positive Fuchsian Markoff maps \cite{BowditchMcShane,Bowditch}.

\begin{inputthm}[Fricke identity and Bowditch sink]
\label{thm:fricke-bowditch}
The trace labels $x_w$ have the following two properties.

\begin{itemize}
\item If $u,v\in\Prim$ satisfy $\det(u,v)=1$, then the two third trace labels
across the corresponding edge of $\mathcal T$ are $x_{u+v}$ and $x_{u-v}$.
They are the two roots of
$$
Z^2-x_ux_vZ+x_u^2+x_v^2=0.
$$
Equivalently,
$$
x_{u+v}+x_{u-v}=x_ux_v,
\qquad
x_{u+v}x_{u-v}=x_u^2+x_v^2.
$$
\item Orient each edge of $\mathcal T$ toward the endpoint whose third trace
label is smaller. If the two third trace labels are equal, orient the edge
both ways. Since the labels come from a positive Fuchsian once-punctured-torus
representation, Bowditch's sink theorem gives a vertex
$$
s_X\in\mathcal T
$$
such that every edge incident to $s_X$ has at least one orientation pointing
toward $s_X$, and every edge not incident to $s_X$ has at least one orientation
pointing toward $s_X$ along the unique path to $s_X$.
\end{itemize}
\end{inputthm}

\subsection{Farey gaps}

The connected components of
$
\mathbb P^1(\mathbb R)\setminus \ProjPrim_H
$
are open projective intervals. We call them the height-$H$ Farey gaps.

The next lemma records the elementary fact that the endpoints of a height-$H$ Farey gap
are Farey neighbors.

\begin{lemma}
\label{lem:farey-gap}
Let $J$ be a height-$H$ Farey gap. Then its endpoints have primitive
representatives $u,v\in\Prim$ such that one lift of $J$ is
$
\{su+tv:s>0,\ t>0\},
$
and, after possibly interchanging $u$ and $v$,
$
\det(u,v)=1.
$
Moreover, the projective class of $u+v$ lies in $J$ and satisfies
$
\height(u+v)>H.
$
\end{lemma}

\begin{proof}
Choose a lift of $J$ to the circle of oriented rays, and choose primitive
vectors $u$ and $v$ on its two endpoint rays. Choose their signs so that the
lift is the cone
$$
\{su+tv:s>0,\ t>0\}.
$$
After interchanging $u$ and $v$, assume $\det(u,v)>0$. We show that
$\det(u,v)=1$.

Suppose $\det(u,v)>1$. Then the lattice $\mathbb Zu+\mathbb Zv$ has index
greater than $1$ in $\mathbb Z^2$. Hence the fundamental parallelogram spanned
by $u$ and $v$ contains an integer point
$$
q=au+bv
$$
not belonging to $\mathbb Zu+\mathbb Zv$, with $0\leq a,b<1$. Since $u$ and
$v$ are primitive, $q$ cannot lie on either boundary ray. Replacing $q$ by
$u+v-q$ if necessary, we may assume $a+b\leq 1$. Thus $q$ lies in the triangle
with vertices $0,u,v$, and its projective class lies in the open gap $J$.

Let $q_0$ be the primitive vector in the direction of $q$. Then
$$
\height(q_0)\leq \height(q)
\leq
\max\{\height(u),\height(v)\}
\leq H.
$$
This contradicts the definition of $J$, since $J$ contains no projective
primitive class of height at most $H$. Therefore $\det(u,v)=1$.

Now $u+v$ is primitive, because
$
\det(u,u+v)=1.
$
It lies in the open cone between $u$ and $v$, so its projective class lies in
$J$. Since $J$ contains no primitive projective class of height at most $H$, we
get
$
\height(u+v)>H.
$
\end{proof}

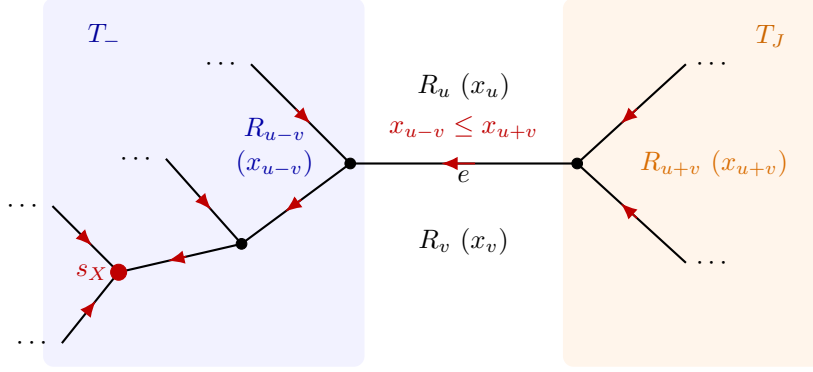
\begin{figure}[h]
\centering
\begin{tikzpicture}[scale=1.25,>=Latex,font=\small]

\newcommand{\orientededge}[2]{%
  \draw[thick] (#1) -- (#2);
  \draw[-{Latex[length=2.2mm,width=1.8mm]},
        red!75!black,
        line width=0.7pt]
        ($(#1)!0.45!(#2)$) -- ($(#1)!0.60!(#2)$);
}

\coordinate (L) at (-1.2,0);
\coordinate (R) at (1.2,0);
\coordinate (A) at (-2.35,-0.85);
\coordinate (S) at (-3.65,-1.15);

\coordinate (Lup) at (-2.25,1.05);
\coordinate (Aup) at (-3.15,0.05);
\coordinate (Sup) at (-4.35,-0.45);
\coordinate (Sdown) at (-4.25,-1.9);
\coordinate (Rup) at (2.35,1.05);
\coordinate (Rdown) at (2.35,-1.05);

\fill[black!5,rounded corners] (-4.45,-2.15) rectangle (-1.05,1.75);
\fill[blue!8,rounded corners] (1.05,-2.15) rectangle (3.8,1.75);

\node[black!60!black] at (-3.8,1.35) {$T_-$};
\node[blue!80!black] at (3.25,1.35) {$T_J$};

\orientededge{R}{L}
\orientededge{L}{A}
\orientededge{A}{S}

\orientededge{Lup}{L}
\orientededge{Aup}{A}
\orientededge{Sup}{S}
\orientededge{Sdown}{S}

\orientededge{Rup}{R}
\orientededge{Rdown}{R}

\node[left] at (Lup) {$\cdots$};
\node[left] at (Aup) {$\cdots$};
\node[left] at (Sup) {$\cdots$};
\node[left] at (Sdown) {$\cdots$};
\node[right] at (Rup) {$\cdots$};
\node[right] at (Rdown) {$\cdots$};

\foreach \P in {L,R,A}
  \filldraw[black] (\P) circle (1.6pt);

\filldraw[red!75!black] (S) circle (2.4pt);
\node[left,red!75!black] at (S) {$s_X$};

\node[above] at (0,-0.3) {$e$};

\node at (0,0.82) {$R_u\ (x_u)$};
\node at (0,-0.82) {$R_v\ (x_v)$};
\node[black!65!black] at (-2,0.35) {$R_{u-v}$};
\node[black!65!black] at (-2,0) {$(x_{u-v})$};
\node[blue!85!black] at (2.65,0) {$R_{u+v}\ (x_{u+v})$};

\node[
  red!75!black,
  fill=white,
  inner sep=2pt
] at (0,0.38) {$x_{u-v}\leq x_{u+v}$};

\end{tikzpicture}
\caption{An illustration for Lemma \ref{lem:large-root}. The edge $e$
separates the gap side $T_J$, adjacent to $R_{u+v}$, from the sink side $T_-$.
The arrows point toward the Bowditch sink $s_X$.}
\label{fig:large-root-gap-side}
\end{figure}

We next combine the Farey-neighbor property with Bowditch's sink theorem. Let
$r_1,r_2,r_3\in\Prim/\{\pm 1\}$ be the three labels of the complementary regions of
$\mathcal T$ incident to the Bowditch sink $s_X$. Define
$$
H_0  \overset{\mathrm{def}}{=} \max\{\height(r_1),\height(r_2),\height(r_3)\}.
$$

\begin{lemma}[The gap side carries the larger Fricke root]
\label{lem:large-root}
Let $H\geq H_0$, and let $J$ be a height-$H$ Farey gap with endpoint
representatives $u,v$ chosen as in Lemma \ref{lem:farey-gap}. Let $e$ be the
edge of $\mathcal T$ corresponding to the Farey pair $\{u,v\}$. Then the
component of $\mathcal T\setminus\{e\}$ facing $J$ is the component adjacent
to the complementary region $R_{u+v}$, and
$$
x_{u+v}\geq x_{u-v}.
$$
\end{lemma}

\begin{proof}
In this proof, a primitive vector is also used to denote its projective class.
Across the edge $e$, the two third labels are $u+v$ and $u-v$. Let $T_J$ be
the component of $\mathcal T\setminus\{e\}$ facing the gap $J$, and let $T_-$
be the other component.

By Lemma \ref{lem:farey-gap}, the signs of $u$ and $v$ were chosen so that
$u+v$ lies in $J$. Hence $T_J$ is the component adjacent to $R_{u+v}$, while
$T_-$ is the component adjacent to $R_{u-v}$.

Since $H\geq H_0$, the three labels incident to the Bowditch sink $s_X$ have
height at most $H$. None of them lies in $J$. By the planar dual-tree picture,
every vertex in $T_J$ is incident to some region labelled by a class in $J$.
Thus $s_X\notin T_J$, and so $s_X\in T_-$.

Therefore Bowditch's orientation of $e$ points toward the $u-v$ side. By the
orientation rule, an edge points toward the endpoint with smaller third trace
label. Hence
$
x_{u-v}\leq x_{u+v}.
$
\end{proof}

\subsection{Support directions and normal turn}
\label{sec:support-turn}

This subsection fixes the convex-geometric notation used later to measure the
boundary of the unit ball. The definitions and lemmas in this
subsection apply to the unit ball of any norm on $\R^2$; for convenience we write
them for the stable norm $\|\cdot\|_X$. The material is standard; see
Rockafellar \cite{RockafellarConvexAnalysis} for
supporting functionals and subdifferentials, and Schneider \cite{SchneiderConvexBodies} for normal
directions, and curvature measures of planar convex bodies.

Let $B_X$ be the unit ball of $\|\cdot\|_X$. Recall the fixed standard
Euclidean inner product $\langle\cdot,\cdot\rangle$ on $\mathbb R^2$. We use it
to identify $(\mathbb R^2)^*=\operatorname{Hom}(\mathbb R^2,\mathbb R)$ with
$\mathbb R^2$: for $\lambda\in(\mathbb R^2)^*$, let
$\lambda^\sharp\in\mathbb R^2$ be the unique vector such that
\begin{equation}\label{eq:LambdaSharp}
   \lambda(y)=\langle\lambda^\sharp,y\rangle
\end{equation}
for every $y\in\R^2$. If $\lambda\neq 0$, define
$$
\nu(\lambda)  \overset{\mathrm{def}}{=} \frac{\lambda^\sharp}{|\lambda^\sharp|},
\qquad
|\lambda^\sharp|=\sqrt{\langle\lambda^\sharp,\lambda^\sharp\rangle}.
$$
We also write
$$
|\lambda|\overset{\mathrm{def}}{=}|\lambda^\sharp|
$$
for the induced Euclidean norm on $(\R^2)^*$. Thus, when $\lambda=1$ is a
supporting line for $B_X$, the vector $\nu(\lambda)$ is its Euclidean outward
unit normal.

\begin{figure}[h]
\centering
\begin{tikzpicture}[scale=5.2,>=Latex,font=\small]

\draw[->,gray!60] (-0.72,0) -- (0.78,0) node[right] {$H_1(S;\mathbb R)$};
\draw[->,gray!60] (0,-0.72) -- (0,0.78);

\path[
  fill=black!6,
  draw=black!65!black,
  thick,
  line join=round
]
(0.5195,0.0000) --
(0.5332,0.0702) --
(0.5479,0.1468) --
(0.5609,0.2323) --
(0.5615,0.3242) --
(0.5441,0.4175) --
(0.5195,0.5195) --
(0.4175,0.5441) --
(0.3242,0.5615) --
(0.2323,0.5609) --
(0.1468,0.5479) --
(0.0702,0.5332) --
(0.0000,0.5195) --
(-0.0618,0.4697) --
(-0.1145,0.4273) --
(-0.1612,0.3893) --
(-0.2040,0.3533) --
(-0.2442,0.3182) --
(-0.2836,0.2836) --
(-0.3182,0.2442) --
(-0.3533,0.2040) --
(-0.3893,0.1612) --
(-0.4273,0.1145) --
(-0.4697,0.0618) --
(-0.5195,0.0000) --
(-0.5332,-0.0702) --
(-0.5479,-0.1468) --
(-0.5609,-0.2323) --
(-0.5615,-0.3242) --
(-0.5441,-0.4175) --
(-0.5195,-0.5195) --
(-0.4175,-0.5441) --
(-0.3242,-0.5615) --
(-0.2323,-0.5609) --
(-0.1468,-0.5479) --
(-0.0702,-0.5332) --
(0.0000,-0.5195) --
(0.0618,-0.4697) --
(0.1145,-0.4273) --
(0.1612,-0.3893) --
(0.2040,-0.3533) --
(0.2442,-0.3182) --
(0.2836,-0.2836) --
(0.3182,-0.2442) --
(0.3533,-0.2040) --
(0.3893,-0.1612) --
(0.4273,-0.1145) --
(0.4697,-0.0618) --
cycle;

\node[black!65!black] at (-0.18,-0.15) {$B_X$};
\node[black!65!black,anchor=west] at (-0.34,-0.61) {$\partial B_X$};

\coordinate (O) at (0,0);
\coordinate (pa) at (0.5195,0.5195);
\coordinate (avec) at (0.69,0.69);

\draw[->,black!45,dashed,thick] (O) -- (avec)
node[above right] {$a$};

\draw[->,black,thick] (O) -- (pa);

\filldraw[black] (pa) circle (0.012);
\node[anchor=south west] at (0.54,0.48) {$p_a$};

\coordinate (pb) at (-0.0618,0.4697);
\coordinate (bvec) at (-0.095,0.72);

\draw[->,black!45,dashed,thick] (O) -- (bvec)
node[above left] {$b$};

\draw[->,black!60,thick] (O) -- (pb)
node[left] {$p_b$};

\draw[blue!85!black,very thick,line cap=round,line join=round]
(0.5195,0.5195) --
(0.4175,0.5441) --
(0.3242,0.5615) --
(0.2323,0.5609) --
(0.1468,0.5479) --
(0.0702,0.5332) --
(0.0000,0.5195) --
(-0.0618,0.4697);

\draw[red!75!black,thick]
(0.30,0.6512) -- (0.75,0.3812)
node[below right] {support line $\lambda=1$};

\draw[->,red!75!black,very thick] (pa) -- (0.6095,0.6695)
node[above left] {$\nu(\lambda)$};

\end{tikzpicture}
\caption{The unit stable-norm ball for the
modular torus. The ray $\mathbb R_{>0}a$ meets $\partial B_X$ at
$p_a=a/\|a\|_X$, where $\lambda=1$ supports $B_X$. The red arrow is the support
direction $\nu(\lambda)$, and the blue arc selects the one-sided supporting
functional determined by the cone spanned by $a$ and $b$.}
\label{fig:support-functional}
\end{figure}
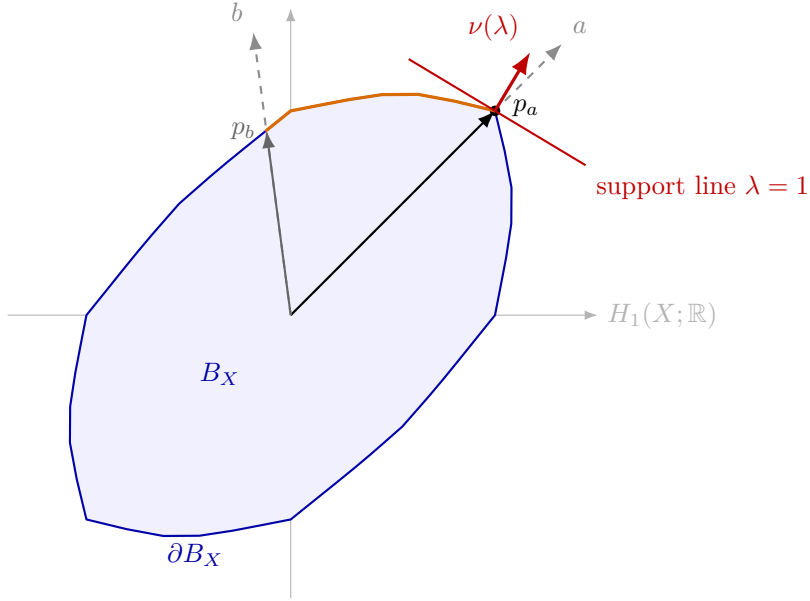

For $a\in\R^2\setminus\{0\}$, set
$$
p_a  \overset{\mathrm{def}}{=} \frac{a}{\|a\|_X}\in\partial B_X.
$$
Define
$$
\partial_X(a)
 \overset{\mathrm{def}}{=}
\{\lambda\in(\R^2)^*:\lambda(y)\leq 1\text{ for all }y\in B_X,
\lambda(p_a)=1\}.
$$
Equivalently, $\lambda\in\partial_X(a)$ if and only if the line $\lambda=1$
supports $B_X$ at $p_a$. By homogeneity, the same set can be written as
\begin{equation}\label{eq:Subdifferential}
    \partial_X(a)
\overset{\mathrm{def}}{=}
\{\lambda\in(\R^2)^*:\lambda(y)\leq \|y\|_X\text{ for all }y\in\R^2,
\lambda(a)=\|a\|_X\}.
\end{equation}
Thus $\partial_X(a)$ is the subdifferential of the convex function
$\|\cdot\|_X$ at $a$.

We also use the normal-turn measure $\Turn$ on $\partial B_X$. This is the
curvature measure of the planar convex curve $\partial B_X$, expressed through
Euclidean outward normal directions. More concretely, as one traverses
$\partial B_X$ counterclockwise, $\Turn$ records the angular measure swept out
by the outward normals $\nu(\lambda)$. On a smooth arc, this is the usual
variation of the outward normal angle. At a corner $p$, the supporting
directions form a nontrivial interval in $S^1$, and
$$
\Turn(\{p\})
$$
is the angular length of this interval. We call $\Turn(\{p\})$ the corner atom
at $p$. If $p$ is not a corner, then $\Turn(\{p\})=0$. Since the outward normal
makes one full turn,
$$
\Turn(\partial B_X)=2\pi.
$$
Thus $\Turn$ is used throughout as a finite Borel measure on $\partial B_X$.
In particular, it is monotone: if $E\subset F\subset \partial B_X$ are Borel
sets, then $\Turn(E)\leq \Turn(F)$.

\begin{lemma}[One-sided derivative of a norm]
\label{lem:directional-derivative}
Let $a\neq 0$ and let $b\in\R^2$. Then
$$
\lim_{t\downarrow 0}\frac{\|a+tb\|_X-\|a\|_X}{t}
=
\max_{\lambda\in\partial_X(a)}\lambda(b).
$$
Moreover, if $t_j\downarrow 0$, if $\mu_j\in\partial_X(a+t_jb)$, and if
$\mu_j\to\mu$, then $\mu\in\partial_X(a)$ and
$$
\mu(b)=\max_{\lambda\in\partial_X(a)}\lambda(b).
$$
\end{lemma}

\begin{proof}
The first formula is the standard directional-derivative formula for finite
convex functions: the one-sided directional derivative is the support function
of the subdifferential; see Rockafellar
\cite[Theorem 23.4]{RockafellarConvexAnalysis}. Applied to the
finite convex function $\|\cdot\|_X$, it gives
$$
\lim_{t\downarrow 0}\frac{\|a+tb\|_X-\|a\|_X}{t}
=
\max_{\lambda\in\partial_X(a)}\lambda(b).
$$

For the final assertion, the closed-graph property of the subdifferential of a
closed proper convex function gives $\mu\in\partial_X(a)$; see Rockafellar
\cite[Theorem 24.4]{RockafellarConvexAnalysis}. Since
$\mu_j\in\partial_X(a+t_jb)$, we have
$$
\|a+t_jb\|_X=\mu_j(a+t_jb)=\mu_j(a)+t_j\mu_j(b).
$$
Also $\mu_j(a)\leq \|a\|_X$. Hence
$$
\frac{\|a+t_jb\|_X-\|a\|_X}{t_j}\leq \mu_j(b).
$$
Passing to the limit gives
$$
\lim_{t\downarrow 0}\frac{\|a+tb\|_X-\|a\|_X}{t}
\leq
\mu(b).
$$
Since $\mu\in\partial_X(a)$, the first formula gives the reverse inequality.
Therefore
$$
\mu(b)=\max_{\lambda\in\partial_X(a)}\lambda(b).
$$
\end{proof}

We now define the supporting functional selected by a second direction. Let
$a,b\in\R^2$ be linearly independent. Define $\lambda_{a|b}$ to be the unique
element of $\partial_X(a)$ such that
\begin{equation}\label{eq:Lambda}
    \lambda_{a|b}(b)
=
\max_{\lambda\in\partial_X(a)}\lambda(b).
\end{equation}
It is unique because all elements of $\partial_X(a)$ agree on $a$, and if two
of them also agree on $b$, then they agree on the basis $a,b$. Geometrically,
if $p_a$ is a corner, the set of supporting functionals is an interval, and
$\lambda_{a|b}$ is the endpoint selected from the side of the cone spanned by
$a$ and $b$. By Lemma \ref{lem:directional-derivative},
\begin{equation}\label{eq:SubDifference}
\lambda_{a|b}(b)
=
\lim_{t\downarrow 0}
\frac{\|a+tb\|_X-\|a\|_X}{t}.
\end{equation}

For linearly independent $a,b\in\mathbb R^2$, define the open cone arc
$$
\partial B_X^\circ(a,b)
\overset{\mathrm{def}}{=}
\left\{
\frac{sa+tb}{\|sa+tb\|_X}:s>0,\ t>0
\right\}
\subset\partial B_X.
$$
Its endpoint support directions are taken in the one-sided sense: namely
$\nu(\lambda_{a|b})$ at $p_a$ and $\nu(\lambda_{b|a})$ at $p_b$. Thus
$\Turn(\partial B_X^\circ(a,b))$ means the normal turn swept along this open
arc, with endpoint corner atoms omitted.

The next elementary estimate converts differences of supporting functionals
into angular differences of Euclidean normal directions. Its only input is that
all supporting functionals lie on the compact dual unit circle.

\begin{lemma}[Functional distance and angular distance]
\label{lem:functional-angle}
There are constants $c_X>0$ and $K_X>0$ such that the following holds. If
$\lambda\in\partial_X(a)$ and $\mu\in\partial_X(b)$ for some $a,b\neq 0$, and
if $\theta$ is the smaller angle between $\nu(\lambda)$ and $\nu(\mu)$, then
$$c_X|\det(\lambda^\sharp,\mu^\sharp)|\le
\theta\leq K_X|\lambda-\mu|
.$$
\end{lemma}

\begin{proof}
Let $\|\cdot\|_{X,*}$ be the dual norm of $\|\cdot\|_X$, defined by
$$
\|\eta\|_{X,*}\overset{\mathrm{def}}{=}\sup_{\|y\|_X\leq 1}|\eta(y)|.
$$
If $\eta$ supports $B_X$ at $p\in\partial B_X$, then, with our normalization, $\eta(y)\leq 1$ for all $y\in B_X$ and $\eta(p)=1$. Since $B_X=-B_X$, this implies $|\eta(y)|\leq 1$ on $B_X$, so $\|\eta\|_{X,*}\leq 1$; the equality $\eta(p)=1$ gives the reverse inequality. Hence $\|\eta\|_{X,*}=1$. Thus all supporting functionals lie on the dual unit
sphere. This sphere is compact and does not contain $0$. Hence there are
constants $0<\kappa_X\leq C_X$ such that
\begin{equation}\label{eq:NormBound}
\kappa_X\leq |\eta|\leq C_X
\end{equation}
for every supporting functional $\eta$.

For supporting functionals $\lambda$ and $\mu$, we have
$$ |\nu(\lambda)-\nu(\mu)|=\left|\frac{\lambda^\sharp}{|\lambda^\sharp|}-\frac{\mu^\sharp}{|\mu^\sharp|}\right|\leq \frac{|\lambda^\sharp-\mu^\sharp|}{|\lambda^\sharp|}+\frac{\big||\mu^\sharp|-|\lambda^\sharp|\big|}{|\lambda^\sharp|}\leq \frac{2}{\kappa_X}|\lambda-\mu|. $$
If $\theta$ is the smaller angle between two unit vectors $u$ and $w$, then $|u-w|=2\sin(\theta/2)$, hence $\theta\leq(\pi/2)|u-w|$. Applying this to $u=\nu(\lambda)$ and $w=\nu(\mu)$ gives $$ \theta\leq \frac{\pi}{\kappa_X}|\lambda-\mu|=K_X|\lambda-\mu|. $$

For the lower bound, use
$$
|\det(\nu(\lambda),\nu(\mu))|
=
\frac{|\det(\lambda^\sharp,\mu^\sharp)|}
{|\lambda^\sharp|\,|\mu^\sharp|}.
$$
By \eqref{eq:NormBound}, the denominator is at most $C_X^2$. Hence
$$
\sin\theta
=
|\det(\nu(\lambda),\nu(\mu))|
\geq
C_X^{-2}|\det(\lambda^\sharp,\mu^\sharp)|.
$$
Since $\theta\geq \sin\theta$ for $0\leq\theta\leq\pi$, after changing the
constant we obtain
$$
\theta\geq c_X|\det(\lambda^\sharp,\mu^\sharp)|.
$$
\end{proof}

\section{Endpoint estimates and the exponential normal-turn tail}
\label{sec:one-gap-turn}
\label{sec:endpoint-estimates}
\label{sec:gap-turn-section}
\label{sec:tail-turn-section}

We prove the main normal-turn estimate in three steps. First, the Fricke
recurrence gives an endpoint estimate for the one-sided supporting lines of a
Farey gap. Second, this estimate shows that one high Farey gap carries very
little normal turn. Finally, summing over all height-$H$ gaps gives the
exponential tail estimate.

\begin{proposition}[Endpoint support estimate]
\label{prop:functional-estimate}
There is a constant $C_X>0$ such that, whenever $u,v\in\Prim$ satisfy
$\det(u,v)=1$ and
$
x_{u+v}\geq x_{u-v},
$
we have
$$
|\lambda_{u|v}(v)-\|v\|_X|
+
|\lambda_{v|u}(u)-\|u\|_X|
\leq
C_Xe^{-\|u\|_X-\|v\|_X}.
$$
\end{proposition}

\begin{proof}
Set
$$
a=\frac{\|u\|_X}{2},
\qquad
b=\frac{\|v\|_X}{2},
\qquad
x=x_u=e^a+e^{-a},
\qquad
y=x_v=e^b+e^{-b},
$$
and put
$$
z=x_{u+v},
\qquad
z'=x_{u-v}.
$$
Then $z\geq z'$, and the Fricke identity (Theorem \ref{thm:fricke-bowditch}) gives
$$
z+z'=xy,
\qquad
zz'=x^2+y^2.
$$
Since simple closed geodesic lengths on $X$ are bounded below, $a,b\geq\eta_X>0$.

Define
$$
A=\frac{z-ye^{-a}}{e^a-e^{-a}},
\qquad
B=\frac{z-xe^{-b}}{e^b-e^{-b}}.
$$
We claim that
$$
|2\log A-2b|+|2\log B-2a|
\leq
C_Xe^{-2a-2b}.
$$
It is enough to prove the estimate for $A$. Since $z+z'=xy$ and $z\geq z'$,
we have $z\geq xy/2$, hence
$$
z-ye^{-a}
\geq
\frac{y}{2}(e^a-e^{-a})>0.
$$
Thus $A>0$. A direct calculation gives
$$
\left(e^{a-b}+e^{b-a}-z\right)
\left(e^{a-b}+e^{b-a}-z'\right)=4.
$$
Moreover,
$$
\frac{xy}{2}-e^{a-b}-e^{b-a}
=
\frac{1}{2}e^{a+b}(1-e^{-2a})(1-e^{-2b})
\geq
c_Xe^{a+b}.
$$
Since $z\geq xy/2$, it follows that
$$
z-e^{a-b}-e^{b-a}\geq c_Xe^{a+b},
$$
and therefore
$$
0<z'-e^{a-b}-e^{b-a}\leq C_Xe^{-a-b}.
$$
Using $z+z'=xy$, we get
$$
e^b-A
=
\frac{z'-e^{a-b}-e^{b-a}}{e^a-e^{-a}},
$$
so
$$
0<1-Ae^{-b}\leq C_Xe^{-2a-2b}.
$$
For $a+b$ large, this gives $Ae^{-b}\geq 1/2$, and hence
$$
|2\log A-2b|
=
2|\log(Ae^{-b})|
\leq
4|1-Ae^{-b}|
\leq
C_Xe^{-2a-2b}.
$$
On the remaining bounded range $a,b\geq\eta_X$, the positivity of $A$ and
compactness allow us to enlarge $C_X$ and keep the same estimate. The estimate
for $B$ follows by interchanging $u$ and $v$.

Now apply the Fricke recurrence along the ray $nu+v$:
$$
x_{(n+1)u+v}+x_{(n-1)u+v}=x x_{nu+v}.
$$
Solving it with initial values $x_v=y$ and $x_{u+v}=z$ gives
$$
x_{nu+v}=Ae^{na}+O_{u,v}(e^{-na}).
$$
Since $A>0$,
$$
\|nu+v\|_X
=
2\arccosh\frac{x_{nu+v}}{2}
=
2na+2\log A+o(1).
$$

By homogeneity and \eqref{eq:SubDifference},
\begin{equation}\label{eq:FanComputing}
    \lambda_{u|v}(v)
=
\lim_{n\to\infty}
\left(\|nu+v\|_X-n\|u\|_X\right)
=
2\log A.
\end{equation}
Hence
$$
|\lambda_{u|v}(v)-\|v\|_X|
=
|2\log A-2b|
\leq
C_Xe^{-2a-2b}
=
C_Xe^{-\|u\|_X-\|v\|_X}.
$$

Similarly, the recurrence along $u+nv$ gives
$
\lambda_{v|u}(u)=2\log B,
$
and therefore
$$
|\lambda_{v|u}(u)-\|u\|_X|
=
|2\log B-2a|
\leq
C_Xe^{-2a-2b}
=
C_Xe^{-\|u\|_X-\|v\|_X}.
$$
\end{proof}

We now pass from endpoint estimates to normal turn over one Farey gap. Let
$J$ be a height-$H$ Farey gap, and choose endpoint representatives $u,v$ as in
Lemma \ref{lem:farey-gap}. Define
$$
\partial B_X^+(J)
\overset{\mathrm{def}}{=}
\left\{\frac{su+tv}{\|su+tv\|_X}:s>0,\ t>0\right\},\quad
\text{ and } \quad
\partial B_X^-(J)  \overset{\mathrm{def}}{=} -\partial B_X^+(J).
$$
Set
$$
\Turn(J)  \overset{\mathrm{def}}{=} \Turn(\partial B_X^+(J))+\Turn(\partial B_X^-(J)).
$$
Endpoint corner atoms at the rays through $u$, $v$, $-u$, and $-v$ are
omitted; corner atoms in the interiors of the two gap arcs are included.

\begin{proposition}[Turn of one Farey gap]
\label{prop:gap-turn}
There are constants $C_X>0$, $\kappa_X>0$, and $H_X\geq 1$ such that the
following holds. Let $H\geq H_X$, and let $J$ be a height-$H$ Farey gap with
endpoints $u,v$ chosen as in Lemma \ref{lem:farey-gap}. Then
$$
\Turn(J)
\leq
C_XH e^{-\|u\|_X-\|v\|_X}
\leq
C_XH e^{-\kappa_X\height(u+v)}.
$$
\end{proposition}

\begin{proof}
Choose $H_X\geq H_0$, where $H_0$ is as in Lemma \ref{lem:large-root}.
Then Lemma \ref{lem:large-root} gives
$
x_{u+v}\geq x_{u-v}.
$
We enlarge $H_X$ later if necessary. Since
$\lambda_{u|v}\in\partial_X(u)$ and $\lambda_{v|u}\in\partial_X(v)$, we have
$$
\lambda_{u|v}(u)=\|u\|_X,
\qquad
\lambda_{v|u}(v)=\|v\|_X.
$$
Thus Proposition \ref{prop:functional-estimate} gives
$$
|\lambda_{u|v}(u)-\lambda_{v|u}(u)|
+
|\lambda_{u|v}(v)-\lambda_{v|u}(v)|
\leq
C_Xe^{-\|u\|_X-\|v\|_X}.
$$
Let $u^*,v^*\in(\mathbb R^2)^*$ be the dual basis to $u,v$, so
$$
u^*(u)=1,\quad u^*(v)=0,\quad v^*(u)=0,\quad v^*(v)=1.
$$
Since $\det(u,v)=1$ and $\height(u),\height(v)\leq H$, the coordinates of
$u^*$ and $v^*$ are bounded by $CH$. Writing
$\eta=\lambda_{u|v}-\lambda_{v|u}$, we have the expansion
$$
\eta=\eta(u)u^*+\eta(v)v^*.
$$
Hence, the previous estimate gives
$$
|\lambda_{u|v}-\lambda_{v|u}| = |\eta|
\leq
CH\bigl(|\eta(u)|+|\eta(v)|\bigr)
\leq
C_XH e^{-\|u\|_X-\|v\|_X}.
$$
By Lemma \ref{lem:functional-angle}, the smaller angle between
$\nu(\lambda_{u|v})$ and $\nu(\lambda_{v|u})$ is at most
$$
C_XH e^{-\|u\|_X-\|v\|_X}.
$$
By the triangle inequality and \eqref{eq:NormComparison},
$$
\|u\|_X+\|v\|_X
\geq
\|u+v\|_X
\geq
\kappa_X\height(u+v)
\geq
 \kappa_X H.
$$
Thus, after increasing $H_X$ if necessary, the angle is less than $1$ for all
$H\geq H_X$. The one-sided endpoint support directions of $\partial B_X^+(J)$ are exactly
$\nu(\lambda_{u|v})$ and $\nu(\lambda_{v|u})$. Along a connected convex
boundary arc, the outward normals sweep one of the two circular intervals
between these endpoint directions. These intervals have lengths $\theta$ and
$2\pi-\theta$, where $\theta$ is the smaller angle. The antipodal arc
$\partial B_X^-(J)$ has the same turn as $\partial B_X^+(J)$; hence each has
turn at most $\pi$. Since $\theta<1<\pi$, the turn of $\partial B_X^+(J)$ must
be the smaller interval, namely $\theta$. Therefore
$$
\Turn(J)
\leq
C_XH e^{-\|u\|_X-\|v\|_X}
\leq
C_XH e^{-\kappa_X\height(u+v)}.
$$
\end{proof}

Summing the one-gap estimate proves the global tail bound in Theorem~\ref{thm:tail-turn}.

\begin{proof}[Proof of Theorem~\ref{thm:tail-turn}]
By Proposition \ref{prop:gap-turn}, there is a cutoff $H_X$ such that the
one-gap estimate holds for all $H\geq H_X$. For $1\leq H<H_X$, the trivial
bound $\Turn(\partial B_X)=2\pi$ is absorbed into the constant. Hence we may
assume $H\geq H_X$.

Since $\ProjPrim_H$ is represented by primitive vectors in
$[-H,H]^2\cap\mathbb Z^2$, we have $\#\ProjPrim_H=O(H^2)$. Thus
$\mathbb P^1(\mathbb R)\setminus\ProjPrim_H$ has $O(H^2)$ height-$H$ Farey
gaps. For each gap $J$, choose endpoint representatives $u,v$ as in
Lemma \ref{lem:farey-gap}. Then
$\height(u+v)>H$, and Proposition \ref{prop:gap-turn} gives
$$
\Turn(J)
\leq
C_XH e^{-\kappa_X\height(u+v)}
\leq
C_XH e^{-\kappa_XH}.
$$
Summing over all height-$H$ gaps gives
$$
\Turn\left(
\partial B_X\setminus
\left\{
\pm\frac{w}{\|w\|_X}:\bar w\in\ProjPrim_H
\right\}
\right)\leq C_XH^3e^{-\kappa_XH}.
$$
Choose
$
\kappa_X'=\frac{\kappa_X}{2}.
$
Then, for every $H\geq 1$,
$$
H^3e^{-\kappa_XH}
=
\left(H^3e^{-\kappa_X'H}\right)e^{-\kappa_X'H}
\leq
\left(\frac{3}{e\kappa_X'}\right)^3 e^{-\kappa_X'H}.
$$
Thus the factor $H^3$ can be absorbed by weakening the exponential rate.
After enlarging $C_X$ and relabeling $\kappa_X'$ as $\kappa_X$, we get the
claimed estimate.
\end{proof}

\begin{remark}
\label{rem:tail-scale-sharpness}
The exponential scale in Theorem~\ref{thm:tail-turn} is sharp up to the value
of the exponent. Indeed, choose $w_H\in\Prim$ with
$
H<\height(w_H)\leq H+1.
$
Since the two corner atoms at
$
\pm w_H/\|w_H\|_X
$
remain in the height-$H$ tail, Proposition~\ref{prop:rational-corner-atoms}
gives
$$
\Turn\left(
\partial B_X\setminus
\left\{\pm\frac{u}{\|u\|_X}:\bar u\in\ProjPrim_H\right\}
\right)
\geq
2\alpha_X(w_H)
\geq
c_Xe^{-C_XH}.
$$
Thus the tail cannot decay faster than exponentially in this formulation. The
precise exponent is not used later.
\end{remark}

\section{Local boundary geometry and the dual-area identity}
\label{sec:local-regularity}

In this section we prove Theorem \ref{thm:intro-local-regularity} and the
dual-area identity in Theorem \ref{thm:dual-area}. The main input is the
exponential turn estimate from the previous section: rational directions of
large height contribute only very small total bending to $\partial B_X$.
We first study the rational corners and use their exact computation to obtain
the dual-area identity, and then turn to the local geometry at irrational
directions.

\subsection{Rational corners and unique supporting lines at irrational directions}
For $w\in\Prim$, define $\alpha_X(w)$ as the corner atom at $w$, equivalently
the exterior angle of $\partial B_X$ at the point $w/\|w\|_X$. We first prove
that rational corner atoms are comparable to
$$
\height(w)e^{-\|w\|_X}.
$$
The exponential tail estimate then implies that the normal-turn measure has no
remaining contribution over irrational directions.

\begin{proposition}[Rational corner atoms]
\label{prop:rational-corner-atoms}
There are constants $c_X,C_X>0$ such that, for every $w\in\Prim$,
$$
c_X\height(w)e^{-\|w\|_X}
\leq
\alpha_X(w)
\leq
C_X\height(w)e^{-\|w\|_X}.
$$
Consequently, there are constants $c_X,C_X,\kappa_X>0$ such that
$$
c_Xe^{-C_X\height(w)}
\leq
\alpha_X(w)
\leq
C_Xe^{-\kappa_X\height(w)}.
$$
\end{proposition}

\begin{proof}
We first prove the lower bound. Choose $u\in\Prim$ with $\det(w,u)=1$. Put
$$
L=\|w\|_X,
\qquad
r=e^{L/2}.
$$
Let $\lambda=\lambda_{w|u}$ and $\mu=\lambda_{w|-u}$ be the two one-sided
supporting functionals at $w$, selected by the two sides $u$ and $-u$. As in the fan computation in the proof of Proposition
\ref{prop:functional-estimate}, especially \eqref{eq:FanComputing}, solving
the Fricke recurrence along the two rays $nw+u$ and $nw-u$ gives
$$
\lambda(u)=2\log A_+,
\qquad
\mu(u)=-2\log A_-,
$$
where
$$
A_+
=
\frac{x_{w+u}-r^{-1}x_u}{r-r^{-1}},
\qquad
A_-
=
\frac{x_{w-u}-r^{-1}x_u}{r-r^{-1}}.
$$
The Fricke identities (Theorem \ref{thm:fricke-bowditch})
$$
x_{w+u}+x_{w-u}=x_wx_u,
\qquad
x_{w+u}x_{w-u}=x_w^2+x_u^2.
$$
give
$$
A_+A_-
=
\frac{(x_{w+u}-r^{-1}x_u)(x_{w-u}-r^{-1}x_u)}{(r-r^{-1})^2}
=
\frac{x_w^2}{(r-r^{-1})^2}
=
\coth^2\frac{L}{2}.
$$

Hence
\begin{equation}\label{eq:Difference}
\lambda(u)-\mu(u)
=
2\log A_+ +2\log A_-
=
4\log\coth\frac{L}{2}.    
\end{equation}
By multiplicativity of determinants, for any $\lambda,\mu\in(\mathbb R^2)^*$
and any $w,u\in\mathbb R^2$,
$$
\det
\begin{pmatrix}
\lambda(w) & \lambda(u)\\
\mu(w) & \mu(u)
\end{pmatrix}
=\det\bigl((\lambda^\sharp\ \mu^\sharp)^T(w\ u)\bigr)
=
\det(\lambda^\sharp,\mu^\sharp)\det(w,u).
$$
Since $\det(w,u)=1$, the two remaining determinants are equal. Moreover,
$\lambda,\mu\in\partial_X(w)$ implies $\lambda(w)=\mu(w)=L$. Therefore, using
\eqref{eq:Difference},
$$
|\det(\lambda^\sharp,\mu^\sharp)|
=
\left|
\det
\begin{pmatrix}
\lambda(w) & \lambda(u)\\
\mu(w) & \mu(u)
\end{pmatrix}
\right|
=
L|\lambda(u)-\mu(u)|
=
4L\log\coth\frac{L}{2}.
$$
 By definition in \eqref{eq:Lambda}, $\lambda=\lambda_{w|u}$ and $\mu=\lambda_{w|-u}$ are the two
supporting functionals in $\partial_X(w)$ with maximal and minimal value on
$u$, respectively. Hence they give the two endpoint normal directions at
$p_w$. Therefore the acute angle between $\nu(\lambda)$ and $\nu(\mu)$ is the corner
atom $\alpha_X(w)$. By Lemma \ref{lem:functional-angle} and the determinant computation above,
\begin{equation}\label{eq:Alpha}
    \alpha_X(w)
\geq
c_XL\log\coth\frac{L}{2}.
\end{equation}

We use the elementary inequality
$$
\log\coth(L/2)\geq 2e^{-L}
$$
valid for all $L>0$. Combining with \eqref{eq:NormComparison} and \eqref{eq:Alpha}, we have
$$
\alpha_X(w)\geq c_XLe^{-L}
\geq c_X\height(w)e^{-L}.
$$

For the upper bound, let $H_X$ be the cutoff in Proposition
\ref{prop:gap-turn}, and put $h=\height(w)$. The finitely many cases
$h\leq H_X$ are absorbed by enlarging $C_X$, so assume $h\geq H_X+1$.
Let $J$ be the height-$(h-1)$ Farey gap containing the direction of $w$.
Its endpoints can be represented by some $u,v\in\Prim$ with
$$
\det(u,v)=1,
\qquad
w=u+v.
$$
The atom $\alpha_X(w)$ is contained in the turn of this gap. Since
$h-1\geq H_X$, Proposition \ref{prop:gap-turn} gives
$$
\alpha_X(w)
\leq
\Turn(J)
\leq
C_X(h-1) e^{-\|u\|_X-\|v\|_X}
\leq
C_Xhe^{-\|w\|_X}.
$$

Finally, applying the norm comparison \eqref{eq:NormComparison} to
$\|w\|_X$ gives constants $c_X,C_X,\kappa_X>0$ such that
$$
c_Xe^{-C_X\height(w)}
\leq
\alpha_X(w)
\leq
C_Xe^{-\kappa_X\height(w)}.
$$
\end{proof}

\begin{remark}
\label{rem:explicit-computability-corner-angle}
The fan computation in the proof gives more than the estimate needed below.
After choosing $u\in\Prim$ with $\det(w,u)=1$, it computes the two endpoint
supporting functionals
$
\lambda=\lambda_{w|u},
\mu=\lambda_{w|-u}
$
from the trace labels $x_u,x_{w+u},x_{w-u}$ and the length $\|w\|_X$. Hence,
under the fixed Euclidean identification $(\mathbb R^2)^*\cong\mathbb R^2$,
the corner angle is
$$
\alpha_X(w)
=
\arccos
\frac{\langle\lambda^\sharp,\mu^\sharp\rangle}
{|\lambda^\sharp||\mu^\sharp|}.
$$
Thus $\alpha_X(w)$ is explicitly computable from these trace labels and the
fixed integral coordinates of $w$ and $u$. The value is independent of the
auxiliary choice of $u$. The exact determinant computation is used in
Theorem~\ref{thm:dual-area}; the explicit arccos formula is not needed below.
\end{remark}

The exact computation at the rational corners also gives the dual-area identity
stated in the introduction.

\begin{proof}[Proof of Theorem~\ref{thm:dual-area}]
Fix a primitive vector $w$ and put
$$
L=\|w\|_X.
$$
Let $\lambda_w^+$ and $\lambda_w^-$ be the two supporting functionals
corresponding to the two sides of the corner at
$
w/\|w\|_X.
$
By duality, the segment
$$
E_w=[(\lambda_w^-)^\sharp,(\lambda_w^+)^\sharp]
$$
is the edge of $B_X^*$ corresponding to this corner. The computation in the proof of
Proposition~\ref{prop:rational-corner-atoms} gives
$$
\left|
\det\left(
(\lambda_w^+)^\sharp,
(\lambda_w^-)^\sharp
\right)
\right|
=
4L\log\coth\left(\frac{L}{2}\right).
$$
Hence the triangle formed by $E_w$ and the origin has area
$$
2L\log\coth\left(\frac{L}{2}\right),
$$
see Figure \ref{fig:dual-triangulation}. The opposite corner gives the opposite edge $-E_w$ and a triangle of the
same area. Thus the pair of antipodal corners determined by $w$ contributes
$$
4L\log\coth\left(\frac{L}{2}\right).
$$

It remains to see that these triangles fill $B_X^*$ up to a set of area zero.
For $H\geq1$, additivity of the normal-turn measure gives
$$
2\pi
=
2\sum_{\substack{\bar w\in\Prim/\{\pm1\}\\ \height(\bar w)\leq H}}
\alpha_X(w)
+
\Turn\left(
\partial B_X\setminus
\left\{
\pm\frac{w}{\|w\|_X}:\height(\bar w)\leq H
\right\}
\right).
$$
By Theorem~\ref{thm:tail-turn}, the last term tends to $0$ as
$H\to\infty$. Therefore
$$
\sum_{\bar w\in\Prim/\{\pm1\}}\alpha_X(w)=\pi.
$$
Thus the rational corners account for the full normal turn of
$\partial B_X$.

Under duality, the interval of supporting directions at a rational corner is
exactly the angular interval covered by the triangle over the corresponding
edge $E_w$. These intervals have disjoint interiors. Together with their
antipodal intervals, the angle sum above shows that they cover all directions
except for a set of angular measure zero. Hence the associated triangles have
disjoint interiors and fill $B_X^*$ up to a set of area zero. Since $B_X^*$ is
bounded, the omitted set has area zero. Summing the areas of the antipodal pairs over primitive classes modulo sign,
and using
$
L=\ell_X(\gamma_w),
$
we obtain the identity in Theorem \ref{thm:dual-area}.
\end{proof}

\begin{proposition}
\label{prop:irrational-differentiability}
At each boundary point lying over an irrational direction, $\partial B_X$
has a unique supporting line.
\end{proposition}

\begin{proof}
The angle-sum argument in the proof of Theorem~\ref{thm:dual-area} shows that
the rational corners account for the entire normal turn of $\partial B_X$.
Hence an irrational boundary point has zero corner angle. For a planar convex
body, a boundary point has more than one supporting line exactly when it is a
corner. Therefore every irrational boundary point has a unique supporting
line.
\end{proof}

\subsection{Flatness order at irrational directions.}

Before quantifying the flatness order at irrational directions, we introduce
some local notions that will be used throughout this section. We first recall
the following graph notion of flatness, following McShane--Rivin
\cite{McShaneRivinTori}.

\begin{definition}[Graph flatness order]
\label{def:graph-flatness}
Let $p\in\partial B_X$ be a boundary point with a unique supporting line.
Choose Euclidean coordinates centered at $p$ such that this supporting line is
$y=0$ and $B_X$ lies locally in $y\geq0$. Write the boundary locally as
$$
y=f(x),
\qquad
f(0)=0.
$$
\begin{itemize}
    
\item For $N>0$, we say that $\partial B_X$ is flat to order at least $N$ at $p$ if
there exists $\delta>0$ such that
$$
f(x)\leq |x|^N
\qquad
\text{whenever } |x|<\delta.
$$

\item The flatness order at $p$ is
$$
\mathfrak f_X(p)
=
\sup\{N>0:\partial B_X\text{ is flat to order at least }N\text{ at }p\},
$$
with value $\infty$ if the set is unbounded.

\item For an irrational direction $\beta$, define
$$
\mathfrak f_X(\beta)=\mathfrak f_X(p)
$$
for either point $p\in\beta\cap\partial B_X$. This is well-defined because
the two points of $\beta\cap\partial B_X$ are antipodal and $B_X$ is centrally
symmetric.

\end{itemize}
\end{definition}
Fix the standard angular distance $d_{\mathbb P^1}$ on $\mathbb P^1(\mathbb R)$.
Let
\begin{equation}\label{eq:radialProjection}
   \pi:\partial B_X\to \mathbb P^1(\mathbb R),
\qquad
\pi(p)=[p], 
\end{equation}

be the radial projection. For an irrational direction $\beta$ and
$\rho>0$, set
\begin{equation}\label{eq:Gamma}
    \Gamma_\beta(\rho)
=
\pi^{-1}
\{\theta\in\mathbb P^1(\mathbb R):d_{\mathbb P^1}(\theta,\beta)<\rho\}.
\end{equation}
For $\rho$ sufficiently small, $\Gamma_\beta(\rho)$ is the disjoint union of
two antipodal connected boundary arcs.

Define the first-entry height by
\begin{equation}\label{eq:fristEntry}
H_\beta(\rho)
=
\min\{\height(\bar w):\bar w\in\Prim/\{\pm1\},
0<d_{\mathbb P^1}(\bar w,\beta)<\rho\}.
\end{equation}

Define the exponential approximation exponent at an irrational direction $\beta$ by
\begin{equation}\label{eq:Omega}
\Omega(\beta)
=
\limsup_{\bar w\to\beta}
\frac{\log(1/d_{\mathbb P^1}(\bar w,\beta))}
{\height(\bar w)}.
\end{equation}
The value of $\Omega(\beta)$ is unchanged if $d_{\mathbb P^1}$ is replaced by
any smooth projective metric.

For $\theta\in\mathbb P^1(\mathbb R)$, choose any nonzero vector
$v\in\theta$, and define the radial value
$$
\tau_X(\theta)
=
\frac{\|v\|_X}{\|v\|_\infty}.
$$
Here $
\|v\|_\infty=\max\{|v_1|,|v_2|\}$ for $v=(v_1,v_2).$
The ratio is independent of the choice of $v$, because both norms are
homogeneous and centrally symmetric. Hence $\tau_X$ is a continuous positive
function on $\mathbb P^1(\mathbb R)$. Equivalently, if $q \in\theta \cap \partial B_X$, then
$$
\tau_X(\theta)=\frac{1}{\|q\|_\infty}.
$$

\begin{proposition}[Localized turn near an irrational direction]
\label{prop:localized-turn-irrational}
Let $\beta$ be an irrational direction. For every $\epsilon>0$, there are
constants $C_{X,\beta,\epsilon}>0$ and $\rho_\epsilon>0$ such that, for all
$0<\rho<\rho_\epsilon$,
$$
\Turn(\Gamma_\beta(\rho))
\leq
C_{X,\beta,\epsilon}
e^{-(\tau_X(\beta)-\epsilon)H_\beta(\rho)}.
$$
\end{proposition}

\begin{proof}
Let $H=H_\beta(\rho)-1$. By definition of $H_\beta(\rho)$ in \eqref{eq:fristEntry}, the set
$$
\{\theta:0<d_{\mathbb P^1}(\theta,\beta)<\rho\}
$$
is contained in
the height-$H$ Farey gap containing $\beta$. Let $u_H,v_H$ be the endpoint
representatives of this gap, chosen as in Lemma~\ref{lem:farey-gap}. As $\rho\to0$, we have $H=H_\beta(\rho)-1\to\infty$. The height-$H$ Farey
gap containing $\beta$ shrinks to $\beta$. Hence the projective directions of
$u_H$, $v_H$, and $u_H+v_H$ tend to $\beta$. Since $\tau_X$ is continuous, for
every $\epsilon>0$ there is $\rho_\epsilon>0$ such that, for all
$0<\rho<\rho_\epsilon$,
$$
\tau_X(\overline{u_H+v_H})
\geq
\tau_X(\beta)-\epsilon.
$$
Equivalently,
$$
\|u_H+v_H\|_X
\geq
(\tau_X(\beta)-\epsilon)\|u_H+v_H\|_\infty.
$$
Using Proposition~\ref{prop:gap-turn} and the triangle inequality for
$\|\cdot\|_X$, we obtain
$$
\Turn(\Gamma_\beta(\rho))
\leq
\Turn(J_H)
\leq
C_XH e^{-\|u_H\|_X-\|v_H\|_X}
\leq
C_XH e^{-\|u_H+v_H\|_X}.
$$
Since
$$
\|u_H+v_H\|_\infty=\height(u_H+v_H)\geq H_\beta(\rho),
$$
and $H\leq H_\beta(\rho)$, the preceding estimate gives
$$
\Turn(\Gamma_\beta(\rho))
\leq
C_XH_\beta(\rho)
e^{-(\tau_X(\beta)-\epsilon/2)H_\beta(\rho)}.
$$
After decreasing $\rho_\epsilon$, we may assume
$$
H_\beta(\rho)\leq e^{\epsilon H_\beta(\rho)/2}.
$$
Thus
$$
\Turn(\Gamma_\beta(\rho))
\leq
C_{X,\beta,\epsilon}
e^{-(\tau_X(\beta)-\epsilon)H_\beta(\rho)}.
$$
\end{proof}
The following proposition computes the lower logarithmic decay rate of the local turn
near an irrational direction in terms of its exponential rational approximation
rate.

\begin{proposition}
\label{prop:exact-local-turn-exponent}
Let $\beta$ be an irrational direction. Then
$$
\liminf_{\rho\to0}
\frac{\log \Turn(\Gamma_\beta(\rho))}{\log\rho}
=
\begin{cases}
\infty, & \Omega(\beta)=0,\\
\dfrac{\tau_X(\beta)}{\Omega(\beta)}, &
0<\Omega(\beta)<\infty,\\
0, & \Omega(\beta)=\infty.
\end{cases}
$$
\end{proposition}

\begin{proof}

For $\rho >0$, density of rational directions and Proposition~\ref{prop:rational-corner-atoms} imply that $\Gamma_\beta(\rho)$ contains a positive corner atom, hence $\Turn(\Gamma_\beta(\rho))>0$ and $\log\Turn(\Gamma_\beta(\rho))$ is well-defined.

We first prove the lower bound. Assume $\Omega(\beta)<\infty$. Fix
$Q>\Omega(\beta)$ and $0<\delta<\tau_X(\beta)$. By the definition of
$\Omega(\beta)$ in \eqref{eq:Omega}, for all sufficiently small $\rho$,
$$
H_\beta(\rho)\geq \frac{1}{Q}\log\frac1\rho.
$$
Proposition~\ref{prop:localized-turn-irrational} gives
$$
\Turn(\Gamma_\beta(\rho))
\leq
C_{X,\beta,\delta}
e^{-(\tau_X(\beta)-\delta)H_\beta(\rho)}.
$$
Taking logarithms and dividing by $\log\rho<0$, we obtain
$$
\liminf_{\rho\to0}
\frac{\log\Turn(\Gamma_\beta(\rho))}{\log\rho}
\geq
\frac{\tau_X(\beta)-\delta}{Q}.
$$
If $\Omega(\beta)=0$, letting $Q\downarrow0$ gives an infinite liminf. If
$0<\Omega(\beta)<\infty$, letting $Q\downarrow\Omega(\beta)$ and then
$\delta\downarrow0$ gives
$$
\liminf_{\rho\to0}
\frac{\log\Turn(\Gamma_\beta(\rho))}{\log\rho}
\geq
\frac{\tau_X(\beta)}{\Omega(\beta)}.
$$

If $\Omega(\beta)=\infty$, the convergence
$\Turn(\Gamma_\beta(\rho))\to0$ as $\rho \to 0$ implies that the quotient is eventually
nonnegative, so the lower bound is $0$.

It remains to prove the reverse inequality when $0<\Omega(\beta)\leq\infty$.
Choose rational directions $\bar w_k\to\beta$ such that
$$
A_k:=
\frac{\log(1/d_{\mathbb P^1}(\bar w_k,\beta))}
{\height(\bar w_k)}
\xrightarrow{k\to\infty}
\Omega(\beta).
$$
Set
$
\rho_k=2d_{\mathbb P^1}(\bar w_k,\beta)
$, then $\bar w_k \in \Gamma_\beta(\rho_k)$, so
$$
\Turn(\Gamma_\beta(\rho_k))\geq \alpha_X(w_k),
$$
where $w_k$ is a primitive representative of $\bar w_k$.
Since $\bar w_k\to\beta$ and $\beta$ is irrational, we have $\height(\bar w_k)\to\infty$.
By Proposition~\ref{prop:rational-corner-atoms},
$$
\alpha_X(w_k)
\geq c_X \height(w_k) e^{-\|w_k\|_X} = 
c_X \height( w_k) e^{-\tau_X( w_k) \height( w_k)}.
$$
Fix $\delta>0$. Since $\tau_X$ is continuous, for all large $k$,
$$
\tau_X( w_k)\leq \tau_X(\beta)+\frac{\delta}{2}.
$$
Also, since $\height( w_k)\to\infty$, for all large $k$,
$$
c_X\height( w_k)\geq e^{-\delta \height(w_k)/2}.
$$Hence
$$
\alpha_X(w_k)
\geq
e^{-(\tau_X(\beta)+\delta)\height(\bar w_k)}.
$$
Since
$
\log\rho_k
<0$ for all large $k$, we have
$$
\frac{\log\Turn(\Gamma_\beta(\rho_k))}{\log\rho_k}
\leq
\frac{\log\alpha_X(w_k)}{\log\rho_k}
\leq
\frac{(\tau_X(\beta)+\delta)\height(\bar w_k)}
{A_k\height(\bar w_k)-\log 2}.
$$
Hence
$$
\liminf_{\rho\to0}
\frac{\log\Turn(\Gamma_\beta(\rho))}{\log\rho}
\leq
\liminf_{k\to\infty}
\frac{\tau_X(\beta)+\delta}
{A_k-\dfrac{\log 2}{\height(\bar w_k)}}.
$$
If $0<\Omega(\beta)<\infty$, letting $\delta\downarrow0$ gives the reverse
inequality. If $\Omega(\beta)=\infty$, the right-hand side is $0$.
\end{proof}

The next lemma records the elementary local estimates used in the flatness
argument: the graph parameter is comparable to projective distance, small
normal turn keeps the graph close to its tangent line, and a nearby corner
forces a definite height above the tangent line.

\begin{lemma}[Local graph-scale and height estimates]
\label{lem:convex-graph-height-estimates}
Let $\beta$ be an irrational direction and $p\in \beta \cap \partial B_X$. Choose graph coordinates at $p$ and write $f$ as in
Definition~\ref{def:graph-flatness}. For $|x|$ sufficiently small, let $\gamma(x)\in\partial B_X$ be the point
whose graph coordinates are $(x,f(x))$. Then the following hold.

\begin{itemize}
    \item 
 There exist constants $c_\beta,C_\beta,\eta>0$ such that, whenever
$|x|<\eta$,
$$
c_\beta |x|
\leq
d_{\mathbb P^1}(\pi(\gamma(x)),\beta)
\leq
C_\beta |x|.
$$

\item If $A_x$ is the graph arc from $\gamma(0)$ to $\gamma(x)$, then
$$
0\leq \frac{f(x)}{|x|}
\leq
2\Turn(A_x)
$$
for all sufficiently small $x\neq0$.

\item If $f$ has a corner at $x_0\neq0$ sufficiently close to $0$, with
exterior angle $\alpha$, and if
$
t=\frac{3x_0}{2}
$ (see Figure \ref{fig:corner-forces-height}),
then
$$
f(t)\geq \frac{1}{2}\alpha |x_0|.
$$

\end{itemize}

\end{lemma}
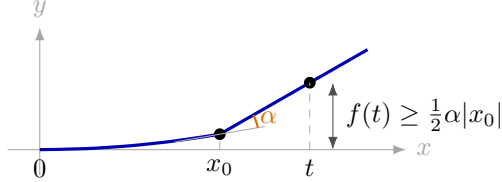
\begin{figure}[!htbp]
\centering
\begin{tikzpicture}[scale=1.7,>=Latex,font=\small]


\coordinate (O) at (0,0);
\coordinate (C) at (1.4,0.120);
\coordinate (T) at (2.1,0.520587);

\draw[->,gray!70] (-0.25,0) -- (2.85,0) node[right] {$x$};
\draw[->,gray!70] (0,-0.20) -- (0,0.95) node[above] {$y$};

\draw[black!70!black,very thick,smooth]
plot coordinates {
(0,0)
(0.35,0.0075)
(0.70,0.0300)
(1.05,0.0675)
(1.40,0.1200)
};
\draw[black!70!black,very thick]
(1.40,0.1200) -- (2.55,0.7781);

\fill[black] (C) circle (1.3pt);
\fill[black] (T) circle (1.3pt);

\draw[gray!65,thin] (1.05,0.0600) -- (1.75,0.180);

\draw[blue!85!black,thick]
($(C)+(9.73:0.30)$)
arc[start angle=9.73,end angle=29.78,radius=0.30];
\node[blue!85!black] at (1.78,0.25) {$\alpha$};

\draw[gray!65,dashed] (1.4,0) -- (C);
\draw[gray!65,dashed] (2.1,0) -- (T);

\node[below] at (O) {$0$};
\node[below] at (1.4,0) {$x_0$};
\node[below] at (2.1,0) {$t$};

\draw[<->,black!70] (2.28,0) -- (2.28,0.520587);
\node[right] at (2.30,0.27) {$f(t)\geq \frac{1}{2}\alpha |x_0|$};

\end{tikzpicture}
\caption{A corner of exterior angle $\alpha$ creates a one-sided slope jump.}
\label{fig:corner-forces-height}
\end{figure}

\begin{proof}
We first compare the graph scale with the projective scale. Since the
supporting line at $p$ is unique and horizontal in the graph coordinates,
$$
f(x)=o(|x|).
$$
Write $(x,y)$ for the local graph coordinates centered at $p$. Then the function
$q\mapsto \det(p,q)$, where determinant is taken in the original vector space, has the form
$$
\det(p,q)=ax+by
$$
for some constants $a,b$.
We have $a\neq0$. Indeed, if $a=0$, then $\det(p,q)=0$ for every point $q$ on
the horizontal line $y=0$, which is the supporting line at $p$. Thus this
supporting line would be the line through $0$ and $p$, contradicting that the
origin lies in the interior of $B_X$. Therefore
$$
\det(p,\gamma(x))=ax+bf(x)=ax+o(|x|),
\qquad a\neq0.
$$
It follows that, for all sufficiently small $|x|$,
$$
c|x|\leq |\det(p,\gamma(x))|\leq C|x|.
$$
Since $\gamma(x)\to p\neq0$, and since projective angular distance is locally
comparable to
$$
\frac{|\det(p,\gamma(x))|}{|p|\,|\gamma(x)|},
$$
there exist constants $c_\beta,C_\beta,\eta>0$ such that, whenever $|x|<\eta$,
$$
c_\beta |x|
\leq
d_{\mathbb P^1}(\pi(\gamma(x)),\beta)
\leq
C_\beta |x|.
$$

(2) We prove the graph-height estimate. For a finite convex function, one-sided
derivatives exist and are monotone. Since the tangent slope at $0$ is $0$,
the one-sided tangent angles tend to $0$ near the origin. Thus, for all
sufficiently small $|x|$, all relevant tangent angles have absolute value at
most $\pi/4$.

Assume first that $x>0$. Let $\phi_x$ be the left tangent angle at
$\gamma(x)$. Convexity gives
$$
0\leq \frac{f(x)}{x}\leq f'_-(x)=\tan\phi_x.
$$
The tangent angle at the origin is $0$, and tangent-angle variation equals
normal-angle variation. Hence
$$
0\leq \phi_x\leq \Turn(A_x).
$$
Since $\tan\phi_x\leq2\phi_x$ for $0\leq\phi_x\leq\pi/4$, we get
$$
0\leq \frac{f(x)}{x}\leq2\Turn(A_x).
$$
The case $x<0$ follows by replacing $f(s)$ with $f(-s)$.

(3) Finally, we prove the corner estimate. Replacing $f(s)$ by $f(-s)$ if
necessary, assume $x_0>0$. Convexity gives
$$
f'_-(x_0)\leq f'_+(x_0).
$$
The exterior angle is
$$
\alpha=\arctan f'_+(x_0)-\arctan f'_-(x_0).
$$
Since $\arctan$ is increasing and $1$-Lipschitz,
$$
\alpha\leq f'_+(x_0)-f'_-(x_0).
$$
Because $x_0>0$ and the one-sided derivatives are monotone,
$$
f'_-(x_0)\geq0.
$$
Thus
$$
f'_+(x_0)\geq\alpha.
$$
Also $f(x_0)\geq0$, because the supporting line at the origin is $y=0$ and
the graph lies locally in $y\geq0$. Therefore convexity gives
$$
f(3x_0/2)
\geq
f(x_0)+\frac{x_0}{2}f'_+(x_0)
\geq
\frac{\alpha x_0}{2}.
$$
\end{proof}

\begin{theorem}[Exact irrational flatness order]
\label{thm:exact-irrational-flatness-order}
Let $\beta$ be an irrational direction and $p \in\beta\cap\partial B_X$. Then
$$
\mathfrak f_X(\beta)
=
\begin{cases}
\infty, & \Omega(\beta)=0,\\
1+\dfrac{\tau_X(\beta)}{\Omega(\beta)}, &
0<\Omega(\beta)<\infty,\\
1, & \Omega(\beta)=\infty.
\end{cases}
$$
In particular, $\partial B_X$ is flat to infinite order at $p$ if and
only if $\Omega(\beta)=0$.
\end{theorem}

\begin{proof}
Choose graph coordinates at $p$ and write $f$ as in
Definition~\ref{def:graph-flatness}. For sufficiently small $|x|$, let
$\gamma(x)\in\partial B_X$ be the point whose graph coordinates are
$(x,f(x))$. Let $A_x$ be the boundary arc from $\gamma(0)=p$ to $\gamma(x)$. By the upper bound in the first part of Lemma~\ref{lem:convex-graph-height-estimates}, there is a constant
$C_\beta>0$ such that, for all sufficiently small $|x|$,
$$
A_x\subset \Gamma_\beta(C_\beta |x|),
$$
where $\Gamma_\beta(\cdot)$ is defined in \eqref{eq:Gamma}. Therefore the second part of Lemma~\ref{lem:convex-graph-height-estimates} gives
\begin{equation}\label{eq:TurnInequality}
    0\leq \frac{f(x)}{|x|}
\leq
2\Turn(A_x)
\leq
\Turn(\Gamma_\beta(C_\beta |x|)).
\end{equation}
Set
\[
L_\beta=
\liminf_{\rho\to0}
\frac{\log \Turn(\Gamma_\beta(\rho))}{\log\rho}.
\]
If $L_\beta>0$, then for every $0<s<L_\beta$, there exists $0<\rho_s<1$ such
that, for all $0<\rho<\rho_s$,
\begin{equation}\label{eq:AnotherTurnInequality}
\Turn(\Gamma_\beta(\rho))\leq \rho^s.
\end{equation}

If $\Omega(\beta)=0$, then Proposition~\ref{prop:exact-local-turn-exponent}
gives $L_\beta=\infty$. Let $N>0$ be arbitrary. Choose
$$
s>\max\{0,N-1\}.
$$
Then from \eqref{eq:TurnInequality} and \eqref{eq:AnotherTurnInequality}, for all sufficiently small $|x|$,
$$
0\leq f(x)
\leq
|x|\Turn(\Gamma_\beta(C_\beta |x|))
\leq
C_\beta^s |x|^{1+s}.
$$
Since $1+s>N$, we have
$$
C_\beta^s |x|^{1+s}\leq |x|^N
$$
for all sufficiently small $|x|$. Hence $\partial B_X$ is flat to order at
least $N$ at $p$. Since $N$ was arbitrary,
$$
\mathfrak f_X(\beta)=\infty.
$$

If $0<\Omega(\beta)<\infty$, then Proposition~\ref{prop:exact-local-turn-exponent}
gives
$
L_\beta=\frac{\tau_X(\beta)}{\Omega(\beta)}
>0$. Let
$
1<N<1+L_\beta.
$
Choose $s$ with
$
N-1<s<L_\beta.
$
Then, for all sufficiently small $|x|$,
$$
0\leq f(x)
\leq
|x|\Turn(\Gamma_\beta(C_\beta |x|))
\leq
C_\beta^s |x|^{1+s} \leq |x|^N.
$$
Hence $\partial B_X$ is flat to order at
least $N$ at $p$. Therefore
$$
\mathfrak f_X(\beta)\geq
1+\frac{\tau_X(\beta)}{\Omega(\beta)}.
$$
If $\Omega(\beta)=\infty$, then $L_\beta=0$. The preceding power estimate gives
no positive exponent. Instead, the unique supporting line at $p$ gives
$
f(x)/|x|\to0.
$
Hence
$
f(x)\leq |x|
$
for all sufficiently small $|x|$. Thus $\partial B_X$ is flat to order at
least $1$ at $p$, and so
$$
\mathfrak f_X(\beta)\geq1.
$$

It remains to prove the matching upper bounds. Assume first that
$0<\Omega(\beta)<\infty$, and let
\begin{equation}\label{eq:NTO}
    N>1+\frac{\tau_X(\beta)}{\Omega(\beta)}.
\end{equation}
Choose rational directions $\bar w_k\to\beta$ such that
$$
A_k:=
\frac{\log(1/d_{\mathbb P^1}(\bar w_k,\beta))}
{\height(\bar w_k)}
\xrightarrow{k\to\infty}
\Omega(\beta).
$$
 By the lower bound in the first part of
Lemma~\ref{lem:convex-graph-height-estimates} and the definition of $A_k$, there is $D_\beta>0$ such that, for all large $k$,
\begin{equation}\label{eq:xk}
    |x_k|
\leq
D_\beta d_{\mathbb P^1}(\bar w_k,\beta)
=
D_\beta e^{-A_k\height(\bar w_k)}.
\end{equation}
For each $k$, choose the sign of a primitive representative $w_k$ of
$\bar w_k$ so that
$$
q_k:=\frac{w_k}{\|w_k\|_X} \xrightarrow{k\to\infty}
p,
$$
and let $x_k$ be the horizontal coordinate of $q_k$ in the
graph chart.
Since
$$
\|w_k\|_X=\tau_X(\bar w_k)\height(\bar w_k),
$$
Proposition~\ref{prop:rational-corner-atoms} gives
$$
\alpha_X(w_k)
\geq
c_X\height(\bar w_k)
e^{-\tau_X(\bar w_k)\height(\bar w_k)}.
$$
By continuity of $\tau_X$ and because $\height(\bar w_k)\to\infty$, for every
$\epsilon>0$, there exists $k_\epsilon$ such that 
$$
\alpha_X(w_k)
\geq
e^{-(\tau_X(\beta)+\epsilon)\height(\bar w_k)},
$$
for all $k>k_\epsilon$. By the third part of Lemma~\ref{lem:convex-graph-height-estimates}, applied at the corner
$q_k$, if
$
t_k=\frac{3x_k}{2},
$
then
$$
f(t_k)\geq \frac12\alpha_X(w_k)|x_k|.
$$

Thus
$$
\frac{f(t_k)}{|t_k|^N}
\geq
c_N\alpha_X(w_k)|x_k|^{1-N}.
$$
Because $1-N<0$ and \eqref{eq:xk},
we get
$$
|x_k|^{1-N}
\geq
D_\beta^{1-N}e^{(N-1)A_k\height(\bar w_k)}.
$$
Therefore, for all large $k$,
\begin{equation}\label{eq:ft}
    \frac{f(t_k)}{|t_k|^N}
\geq
c_{N,\beta}
\exp\left(
\left((N-1)A_k-\tau_X(\beta)-\epsilon\right)
\height(\bar w_k)
\right).
\end{equation}
Due to \eqref{eq:NTO}, one can choose $\epsilon>0$ small enough such that
$$
(N-1)(\Omega(\beta)-\epsilon)>\tau_X(\beta)+\epsilon.
$$
Since $A_k>\Omega(\beta)-\epsilon$ for all large $k$, the right-hand side of \eqref{eq:ft}
tends to infinity. Hence, for all large $k$,
$$
f(t_k)>|t_k|^N.
$$
Since $t_k\to0$, this shows that $\partial B_X$ is not flat to order at least
$N$ at $p$. Since this holds for every
$
N>1+\frac{\tau_X(\beta)}{\Omega(\beta)},
$
we get
$$
\mathfrak f_X(\beta)\leq
1+\frac{\tau_X(\beta)}{\Omega(\beta)}.
$$

Finally, assume $\Omega(\beta)=\infty$. Fix any $N>1$, and choose rational
directions $\bar w_k\to\beta$ with $A_k\to\infty$. The same argument, with
$\epsilon=1$, gives
$
\frac{f(t_k)}{|t_k|^N}\to\infty.
$
Hence, for all large $k$,
$
f(t_k)>|t_k|^N.
$
Since $t_k\to0$, the boundary is not flat to order at least $N$ at $p$. This
holds for every $N>1$. Therefore
$
\mathfrak f_X(\beta)=1.
$
\end{proof}

\subsection{Realizing prescribed flatness orders}

The exact flatness formula reduces the realization problem to prescribing the exponential jumps in the continued-fraction denominators of a slope. We show that these jumps can be prescribed inside any interval of projective directions, and hence every flatness order in $[1,\infty]$ occurs in every projective interval. The size of the resulting flatness strata is studied in the next subsection.

\begin{lemma}[A continued-fraction formula]
\label{lem:continued-fractions-slope}
Let $\beta=[(1,\mathfrak b )]$ be an irrational direction, where
$\mathfrak b\in\mathbb R\setminus\mathbb Q$. Let $p_j/q_j$ be the continued-fraction
convergents of $\mathfrak b$, with $q_j>0$. Then
$$
\Omega(\beta)
=
\frac{1}{\max\{1,|\mathfrak b|\}}
\limsup_{j\to\infty}
\frac{\log q_{j+1}}{q_j}.
$$
Consequently, with the convention that a positive number divided by $0$ is
$\infty$ and divided by $\infty$ is $0$,
$$
\mathfrak f_X(\beta)
=
1+
\frac{\|(1,\mathfrak b)\|_X}
{\displaystyle\limsup_{j\to\infty}\frac{\log q_{j+1}}{q_j}}.
$$
\end{lemma}

\begin{proof}
For a reduced rational number $p/q$, with $q>0$, the corresponding primitive
vector in the slope chart is
$
w_{p/q}=(q,p).
$
Hence
$$
\height(w_{p/q})
=
\max\{q,|p|\}
=
q\max\left\{1,\left|\frac{p}{q}\right|\right\}.
$$
For every $\epsilon>0$, if $p/q$ is sufficiently close to $\mathfrak b$, then
$$
(\max\{1,|\mathfrak b|\}-\epsilon)q
\leq
\height(w_{p/q})
\leq
(\max\{1,|\mathfrak b|\}+\epsilon)q.
$$

The slope chart $r\mapsto[(1,r)]$ is locally bi-Lipschitz for the projective
angular distance. Indeed, for $r$ close to $\mathfrak b$,
$$
d_{\mathbb P^1}([(1,\mathfrak b)],[(1,r)])
=
|\arctan \mathfrak b-\arctan r|.
$$
Hence there are constants $c_{\mathfrak b},C_{\mathfrak b}>0$ such that, for all rational $p/q$ sufficiently close to $\mathfrak b$, $$ c_{\mathfrak b}\left|\mathfrak b-\frac{p}{q}\right| \leq d_{\mathbb P^1}\left([(1,\mathfrak b)],\left[\left(1,\frac{p}{q}\right)\right]\right) \leq C_{\mathfrak b}\left|\mathfrak b-\frac{p}{q}\right|. $$

The additive constants $\log c_{\mathfrak b}$ and $\log C_{\mathfrak b}$ disappear after division by
$q$. Therefore
$$
\Omega(\beta)
=
\frac{1}{\max\{1,|\mathfrak b|\}}
\limsup_{\frac{p}{q}\to \mathfrak b}
\frac{\log(1/|\mathfrak b-\frac{p}{q}|)}{q},
$$
where the limsup is taken over reduced rationals $p/q$ with $q>0$. By the standard estimates \cite[Thm~9 and~13]{KhinchinContinuedFractions} and the best-approximation property of convergents
\cite[Thm 16 and 17]{KhinchinContinuedFractions},
one can deduce that
$$
\limsup_{\frac{p}{q}\to \mathfrak b}
\frac{\log(1/|\mathfrak b-\frac{p}{q}|)}{q}
=
\limsup_{j\to\infty}
\frac{\log q_{j+1}}{q_j}.
$$
Finally, the stated formula for $\mathfrak f_X(\beta)$ follows from
Theorem~\ref{thm:exact-irrational-flatness-order} and
$$
\tau_X(\beta)
=
\frac{\|(1,\mathfrak b)\|_X}{\max\{1,|\mathfrak b|\}}.
$$
\end{proof}

\begin{proposition}[Prescribed irrational flatness orders]
\label{prop:prescribed-flatness-orders}
For every nonempty projective interval $U\subset\mathbb P^1(\mathbb R)$ and
every $A\in[1,\infty]$, there exists an irrational direction $\beta\in U$ such
that
$$
\mathfrak f_X(\beta)=A.
$$
\end{proposition}

\begin{proof}
Choose an open slope interval $I\subset\mathbb R$ whose projectivization lies
in $U$. Fix a finite initial block $a_0,\ldots,a_k$ such that every irrational
number
$$
[a_0;a_1,\ldots,a_k,c_{k+1},c_{k+2},\ldots],
\qquad c_i\in\mathbb Z_{\geq1},
$$
lies in $I$. Assume first that $1<A<\infty$. Having chosen the partial quotients up to
$c_j$, let $p_j/q_j$ be the corresponding convergent and put
$$
D_j=\frac{\|(1,p_j/q_j)\|_X}{A-1}.
$$
Since $q_{j+1}=c_{j+1}q_j+q_{j-1}$ (\cite[Thm 1]{KhinchinContinuedFractions}), and since $D_j$ is bounded away from $0$, we may
choose $c_{j+1}$ comparable to $\exp(D_jq_j)/q_j$. Then
$q_{j+1}=\exp(D_jq_j)(1+o(1))$, so
$$
\frac{\log q_{j+1}}{q_j}=D_j+o(1).
$$ The resulting
irrational $\mathfrak b$ lies in $I$, hence $\beta=[(1,\mathfrak b)]\in U$.
Since $p_j/q_j\to\mathfrak b$, continuity of $\|\cdot\|_X$ gives
$$
\limsup_{j\to\infty}\frac{\log q_{j+1}}{q_j}
=
\frac{\|(1,\mathfrak b)\|_X}{A-1}.
$$
Lemma~\ref{lem:continued-fractions-slope} gives $\mathfrak f_X(\beta)=A$.

For $A=\infty$, choose all remaining partial quotients bounded. Then
$q_{j+1}=O(q_j)$, so $\log q_{j+1}/q_j\to0$, and
Lemma~\ref{lem:continued-fractions-slope} gives $\mathfrak f_X(\beta)=\infty$.

For $A=1$, choose infinitely many partial quotients so large that
$q_{j+1}\geq \exp(jq_j)$. Then
$\limsup_j \log q_{j+1}/q_j=\infty$, and
Lemma~\ref{lem:continued-fractions-slope} gives $\mathfrak f_X(\beta)=1$.
\end{proof}
\subsection{Measure, category, and dimension of irrational flatness}

The flatness strata have opposite metric and topological behavior: infinite
flatness is large for Lebesgue measure, while least flatness is large for
Baire category. Set
$$
\mathcal I=\mathbb P^1(\mathbb R)\setminus\mathbb P^1(\mathbb Q),
\qquad
\mathcal F_A=\mathcal F_A(X)=\{\beta\in\mathcal I:\mathfrak f_X(\beta)=A\}.
$$
All topological notions below are relative to $\mathcal I$.

\begin{proposition}
\label{prop:measure-category-dimension-flatness}
For every $A\in[1,\infty]$, the set $\mathcal F_A$ is dense in $\mathcal I$.
Moreover:
\begin{itemize}
\item $\mathcal F_1$ is a dense $G_\delta$ subset of $\mathcal I$ and has
Hausdorff dimension $0$;
\item $\mathcal F_\infty$ has full Lebesgue measure, its complement has
Hausdorff dimension $0$, and $\mathcal F_\infty$ is meagre;
\item for every $1<A<\infty$, the set $\mathcal F_A$ is meagre and has
Hausdorff dimension $0$.
\end{itemize}
\end{proposition}

\begin{proof}
The density of every level follows from
Proposition~\ref{prop:prescribed-flatness-orders}.

For $\epsilon>0$, let
$$
E_\epsilon
=
\left\{
\theta\in\mathbb P^1(\mathbb R):
d_{\mathbb P^1}(\theta,\bar w)<e^{-\epsilon\height(\bar w)}
\text{ for infinitely many }\bar w\in\Prim/\{\pm1\}
\right\}.
$$
There are $O(h)$ projective primitive classes of height $h$. Hence, for every
$s>0$ and $N\geq1$, the tail of $E_\epsilon$ is covered by intervals with
total $s$-content at most
$$
C\sum_{h\geq N}h e^{-s\epsilon h},
$$
which tends to $0$ as $N\to\infty$. Thus $\dim_{\mathrm H}E_\epsilon=0$.
Since
$$
\{\beta\in\mathcal I:\Omega(\beta)>0\}
\subset
\bigcup_{m=1}^{\infty}E_{1/m},
$$
we get
$$
\dim_{\mathrm H}\{\beta\in\mathcal I:\Omega(\beta)>0\}=0.
$$
In particular, $\{\Omega=0\}$ has full Lebesgue measure and its complement has
Hausdorff dimension $0$.

Next we prove the category statement. For integers $m,N\geq1$, set
$$
U_{m,N}
=
\mathcal I
\cap
\bigcup_{\substack{\bar w\in\Prim/\{\pm1\}\\ \height(\bar w)\geq N}}
\left\{
\theta\in\mathbb P^1(\mathbb R):
d_{\mathbb P^1}(\theta,\bar w)<e^{-m\height(\bar w)}
\right\}.
$$
Each $U_{m,N}$ is open and dense in $\mathcal I$. Hence
$$
G=
\bigcap_{m=1}^{\infty}\bigcap_{N=1}^{\infty}U_{m,N}
$$
is dense $G_\delta$ in $\mathcal I$, and by construction
$
G=\{\beta\in\mathcal I:\Omega(\beta)=\infty\}.
$

The exact flatness formula (Theorem \ref{thm:exact-irrational-flatness-order}) gives
$$
\mathcal F_1=\{\Omega=\infty\},
\qquad
\mathcal F_\infty=\{\Omega=0\}.
$$
Therefore $\mathcal F_1$ is dense $G_\delta$ and has Hausdorff dimension $0$,
while $\mathcal F_\infty$ has full Lebesgue measure and complement of
Hausdorff dimension $0$. Since $\mathcal F_1$ is dense $G_\delta$, its
complement is meagre; hence every $\mathcal F_A$ with $A>1$ is meagre,
including $\mathcal F_\infty$. Finally, every finite level
$\mathcal F_A$, $1\leq A<\infty$, is contained in $\{\Omega>0\}$, so it has
Hausdorff dimension $0$.
\end{proof}

The exact flatness formula gives a rigidity consequence: a dense overlap of
finite flatness levels already determines the marked torus.

\begin{corollary}[Finite flatness levels determine the marked torus]
\label{cor:finite-flatness-level-determines-marked-torus}
Let $A,B\in(1,\infty)$, and let $X,Y$ be marked complete finite-area hyperbolic
once-punctured tori. If
$$
\mathcal F_A(X)\cap\mathcal F_B(Y)
$$
is dense in $\mathcal I$, then $A=B$,
$
\|\cdot\|_X=\|\cdot\|_Y,
$
and $X=Y$ as marked hyperbolic tori. In particular, for every fixed
$A\in(1,\infty)$, the level set $\mathcal F_A(X)$ determines the marked torus
$X$.
\end{corollary}

\begin{proof}
For every $\beta\in\mathcal F_A(X)\cap\mathcal F_B(Y)$, the exact flatness
formula gives
$$
\tau_X(\beta)=(A-1)\Omega(\beta),
\qquad
\tau_Y(\beta)=(B-1)\Omega(\beta).
$$
Since $\Omega(\beta)$ depends only on the fixed integral coordinates, we get
$$
\tau_X(\beta)=c\tau_Y(\beta),
\qquad
c=\frac{A-1}{B-1},
$$
on a dense subset of $\mathcal I$. By continuity,
$$
\tau_X=c\tau_Y
$$
on $\mathbb P^1(\mathbb R)$, and hence
$$
\|\cdot\|_X=c\|\cdot\|_Y.
$$

Thus every simple closed geodesic $\gamma$ satisfies
$$
\ell_X(\gamma)=c\ell_Y(\gamma).
$$
McShane's identity \cite{mcshane1998simple} for once-punctured tori gives
$$
\sum_\gamma \frac{1}{1+e^{\ell_X(\gamma)}}=\frac12
=
\sum_\gamma \frac{1}{1+e^{\ell_Y(\gamma)}}.
$$
If $c>1$, every term in the first sum is smaller than the corresponding term
in the second; if $c<1$, every term is larger. Hence $c=1$. Therefore
$A=B$ and $\|\cdot\|_X=\|\cdot\|_Y$.

Finally, equality of the norm gives equality of the lengths of the three
marked curves represented by a basis $a,b$ and by $a+b$. Hence the positive
Fricke trace triples
$$
(x_a,x_b,x_{a+b})
$$
agree for $X$ and $Y$. The standard Fricke parametrization of marked cusped
once-punctured tori then gives $X=Y$.
\end{proof}
\begin{remark}
    In fact, it is enough that the closure of the overlap contains the vertices of a Farey triangle. In particular, density in any nonempty open interval of $\mathcal{I}$ is enough.
\end{remark}

\section{Lattice counting and proof of Theorem \ref{thm:main}}
\label{sec:lattice-count}
\label{sec:proof-main}

We now convert the exponential normal-turn estimate into a lattice-point bound.
The key input is a localized version of Jarník's chord-counting argument
\cite{Jarnik}. In Jarník's global estimate, the angular range of chord
directions is bounded by an absolute constant and is absorbed into the implied
constant. Here we keep this angular range: if the outward normal directions
along a convex arc $A$ sweep only a small sector, then the tangent directions,
and hence the primitive chord directions between consecutive lattice points,
also lie in a small sector. Since short primitive lattice directions are sparse
in such a sector, many lattice points force many long chords. This gives the
extra factor $\Turn(A)^{1/3}$ in the localized estimate below. We also use the
standard monotonicity of normal directions along convex curves
\cite{SchneiderConvexBodies}.

\begin{lemma}[Lattice points on a small-turn arc]
\label{lem:lattice-sector}
Let $A$ be a connected arc of Euclidean arclength $\Len(A) \le R$ in the boundary of a
strictly convex plane body. Assume that
$
\Turn(A)\leq \theta,$ where $ 0\leq \theta\leq 2\pi.
$
Then there is a universal constant $C>0$ such that
$$
\#(A\cap\mathbb Z^2)
\leq
C\left(R^{2/3}\theta^{1/3}+1\right).
$$
\end{lemma}

\begin{proof}
The proof has three steps: a sector count for primitive directions, its
application to consecutive chords on a convex graph arc, and a reduction of a
general convex arc to at most four small-turn arcs, where the piecewise estimates are combined
by H\"older's inequality.

\textbf{Step 1.} We first count primitive lattice directions in a fixed
angular sector. Let $I$ be an angular
sector of length $\theta$, and let $r\geq 1$. Split $I$ into subintervals
$I_1,\ldots,I_s$, where $s\leq 4$ and $|I_j|\leq \pi/2$. Let $M_j$ be the
number of oriented primitive directions in $I_j$ represented by vectors of
Euclidean norm at most $r$.

Fix $j$, and list these directions as primitive vectors
$z_1,\ldots,z_{M_j}$ in angular order. If $M_j\leq 1$, then
$M_j\leq |I_j|r^2+1$. Otherwise, for consecutive directions, put
$\phi_i=\angle(z_i,z_{i+1})$. Since $z_i$ and $z_{i+1}$ are
distinct primitive lattice directions,
$$
1\leq |\det(z_i,z_{i+1})|
=
|z_i||z_{i+1}|\sin\phi_i
\leq r^2\phi_i.
$$
Thus $\phi_i\geq r^{-2}$ for every $i$. Since the first and last primitive directions need not lie at the endpoints of $I_j$,
$$
|I_j|
\geq
\sum_{i=1}^{M_j-1}\phi_i
\geq
(M_j-1)r^{-2}.
$$
Therefore, if $M$ is the total number of such directions in $I$, then
\begin{equation}\label{eq:SectorCount}
   M\leq \sum_{j=1}^s M_j
\leq r^2\sum_{j=1}^s |I_j|+s
\leq \theta r^2+4
\leq C_0(\theta r^2+1).
\end{equation}

\textbf{Step 2.} We prove the estimate for an arc $A$ with
$\Turn(A)=\theta_A\leq \pi/2$. Let
$P_0,\ldots,P_N$ be the lattice points of $A$, listed in their order along
the arc. If $N\leq 2$, there is nothing to prove. Write
$$
P_{i+1}-P_i=t_i d_i,
$$
where $t_i\in \mathbb Z_{\geq 1}$ and $d_i\in \mathbb Z^2$ is primitive. Since the outward normal directions swept by $A$ lie in an angular interval
of length $\theta_A$, the tangent directions also lie in an angular interval
of length $\theta_A$. By convexity, each chord direction $d_i$ lies between
the tangent directions swept between $P_i$ and $P_{i+1}$. Hence all $d_i$
lie in one angular sector of length at most $\theta_A$. Strict convexity implies that the directions $d_i$ are distinct. Moreover,
\begin{equation}\label{eq:ChordDirectionBound}
   \sum_{i=0}^{N-1}|d_i|
\leq
\sum_{i=0}^{N-1}|P_{i+1}-P_i|
\leq R. 
\end{equation}

Assume $\theta_A>0$; the case $\theta_A=0$ is immediate. Choose
$$
r=\left(\frac{N}{4C_0\theta_A}\right)^{1/2}.
$$
If $r<1$ or $N<4C_0$, then $N$ is universally bounded. Thus we
may assume $r\geq 1$ and $N\geq 4C_0$. Let $M$ be the number of indices $i$ such that $|d_i|\leq r$. Since the
directions $d_i$ are distinct and all lie in an angular sector of length at
most $\theta_A$, the sector count in (\ref{eq:SectorCount}) gives
$$
M\leq C_0(\theta_A r^2+1).
$$
By the choice of $r$,
$$
C_0(\theta_A r^2+1)
=
C_0\left(\frac{N}{4C_0}+1\right)
=
\frac N4+C_0
\leq
\frac N2,
$$
where the last inequality uses $N\geq 4C_0$. Thus at most $N/2$ of the
vectors $d_i$ have length at most $r$.
Consequently at least $N/2$ of the $N$ vectors $d_i$ satisfy $|d_i|>r$.
Therefore, from (\ref{eq:ChordDirectionBound}),
$$
R
\geq
\sum_{i=0}^{N-1}|d_i|
\geq
\frac N2 r.
$$
Hence, for the graph case, we obtain
$$
\#(A\cap\mathbb Z^2) = N+1 \leq \frac{2R}{r}+1  \leq C(R^{2/3}\theta^{1/3}+1).
$$

\textbf{Step 3.} We reduce the general case to the case in Step 2. Orient the arc $A$. By the monotonicity of outward normal directions along a convex curve,
the outward normals swept along $A$ form a circular interval of length
$\Turn(A)\leq\theta$. Split this swept normal interval into at most four angular sectors of length
at most $\pi/2$, and cut $A$ into the corresponding connected subarcs $A_j$.

Set
$$
R_j=\Len(A_j),
\qquad
\theta_j=\Turn(A_j).
$$
Since cutting can only discard endpoint corner atoms, we have
$$
\sum_j R_j \le R,
\qquad
\sum_j \theta_j\leq \Turn(A)\leq\theta.
$$
Applying Step 2 to each $A_j$ and summing gives
$$
\#(A\cap\mathbb Z^2)
\leq
C\sum_j\left(R_j^{2/3}\theta_j^{1/3}+1\right).
$$
By H\"older's inequality,
$$
\sum_j R_j^{2/3}\theta_j^{1/3}
\leq
\left(\sum_j R_j\right)^{2/3}
\left(\sum_j \theta_j\right)^{1/3}
\leq
R^{2/3}\theta^{1/3}.
$$
Since the number of subarcs is at most four, this proves the lemma.
\end{proof}

\begin{proof}[Proof of Theorem \ref{thm:main}]
Let $\kappa_X>0$ be the constant from Theorem \ref{thm:tail-turn}. Fix $L\geq 2$, and let $H\geq 1$ be an integer to be chosen later. We separate the rational boundary points of height at most $H$. These points
may carry corner atoms, while the remaining boundary has exponentially small
normal turn by Theorem \ref{thm:tail-turn}. This is exactly where the localized
Jarník estimate, Lemma \ref{lem:lattice-sector}, improves the count.

Set
$$
P_H  \overset{\mathrm{def}}{=} 
\left\{
\pm\frac{w}{\|w\|_X}:\bar w\in\ProjPrim_H
\right\}
\subset \partial B_X.
$$
Since $P_H$ is indexed, up to a factor of $2$, by primitive vectors in
$[-H,H]^2\cap\mathbb Z^2$, we have
$$
\#P_H\leq C H^2.
$$
The dilation map $p\mapsto Lp$ is a bijection from $P_H$ onto $LP_H$. Hence
\begin{equation}\label{eq:EstimateRPH}
\#(LP_H\cap\mathbb Z^2)
\leq
\#(LP_H)
=
\#P_H
\leq
C H^2.
\end{equation}
Let $A_1,\ldots,A_m$ be the connected components of
$\partial B_X\setminus P_H$. Since $\#P_H\leq C H^2$, we have
$$
m\leq C H^2.
$$
For each $j$, set
$$
R_j=\Len(LA_j),
\qquad
\theta_j=\Turn(A_j),
$$
where endpoint atoms are not counted. By Theorem \ref{thm:tail-turn},
\begin{equation}\label{eq:ThetaSumMain}
\sum_{j=1}^m\theta_j
\leq
C_Xe^{-\kappa_XH}.
\end{equation}
The dilation $x\mapsto Lx$ multiplies Euclidean arclength by $L$ and preserves
outward normal directions. Hence
$$
R_j=\Len(LA_j)=L\Len(A_j),
\qquad
\Turn(LA_j)=\Turn(A_j)=\theta_j.
$$
Since $P_H$ is finite, it has zero arclength, and hence
\begin{equation}\label{eq:LengthSumMain}
\sum_{j=1}^m R_j
=
L\sum_{j=1}^m\Len(A_j)
=
L\Len(\partial B_X)
\leq
C_XL.
\end{equation}
Applying Lemma \ref{lem:lattice-sector} to the arcs $LA_j$ and summing gives
$$
\#\bigl(L(\partial B_X\setminus P_H)\cap\mathbb Z^2\bigr)
\leq
C\sum_{j=1}^m R_j^{2/3}\theta_j^{1/3}
+
Cm.
$$
By H\"older's inequality,
$$
\sum_{j=1}^m R_j^{2/3}\theta_j^{1/3}
\leq
\left(\sum_{j=1}^m R_j\right)^{2/3}
\left(\sum_{j=1}^m\theta_j\right)^{1/3}.
$$
Using \eqref{eq:ThetaSumMain}, \eqref{eq:LengthSumMain}, and $m\leq CH^2$, we
obtain
$$
\#\bigl(L(\partial B_X\setminus P_H)\cap\mathbb Z^2\bigr)
\leq
C_XL^{2/3}e^{-\kappa_XH/3}
+
CH^2.
$$
Together with \eqref{eq:EstimateRPH}, this gives
$$
\#(L\partial B_X\cap\mathbb Z^2)
\leq
C_XL^{2/3}e^{-\kappa_XH/3}
+
2CH^2.
$$
We now choose $H$ so that the exponential term is bounded, namely
$$
H=\left\lceil \frac{2}{\kappa_X}\log L\right\rceil .
$$
Then
$
L^{2/3}e^{-\kappa_XH/3}\leq 1,
$ and $
H\leq C_X\log L,
$
and hence
$$
\#(L\partial B_X\cap\mathbb Z^2)
\leq C_X(\log L)^2.
$$
Every primitive vector $w$ with $\|w\|_X=L$ lies on
$L\partial B_X$. Since $w$ and $-w$ represent the same unoriented simple closed
geodesic,
\begin{equation}\label{eq:MulBound}
  m_X(L)
=
\frac12\#(L\partial B_X\cap\Prim)
\leq
\frac12\#(L\partial B_X\cap\mathbb Z^2)
\leq
C_X(\log L)^2.  
\end{equation}
\end{proof}

\begin{remark}[The active-gap obstruction]
\label{rem:active-gap-loss}
For $H\geq 1$, let $\mathcal A_H(L)$ be the set of height-$H$ Farey gaps
that contain a projective primitive class $\bar w$ with $\|w\|_X=L$. Then the
proof above gives the sharper estimate
$$
m_X(L)
\leq
C_XL^{2/3}e^{-\kappa_XH/3}
+
C_X\#\mathcal A_H(L)
+
C_X\#\{\bar u\in\ProjPrim_H:\|u\|_X=L\}.
$$
Indeed, the last term counts the level-$L$ classes of height at most $H$. All remaining
level-$L$ classes lie in active height-$H$ gaps. Hence the additive $+1$ in
Lemma \ref{lem:lattice-sector} is charged only to those gaps, while the
$R^{2/3}\theta^{1/3}$ terms are summed as before using the tail-turn estimate.

By the upper bound in the norm comparison \eqref{eq:NormComparison}, there is
$\eta_X>0$ such that, if $H\leq \eta_XL$, then every
$\bar u\in\ProjPrim_H$ satisfies $\|u\|_X<L$. Under this condition, the last
term vanishes. Taking $H=\lceil A\log L\rceil$ with any fixed
$A>2/\kappa_X$, the first term is bounded independently of $L$, and the
condition $H\leq \eta_XL$ holds for all sufficiently large $L$. After enlarging
$C_X$ to absorb the remaining bounded range of $L$, we obtain
$$
m_X(L)
\leq
C_X\left(1+\#\mathcal A_{\lceil A\log L\rceil}(L)\right).
$$
Thus, within this approach, the remaining loss is the crude estimate
$$
\#\mathcal A_H(L)\leq C H^2,
$$
which uses only the number of height-$H$ gaps. Any improvement below $(\log L)^2$ therefore requires further inputs.
\end{remark}

\section{Further directions and generalizations}
\label{sec:further-directions}

We have focused on once-punctured tori to keep the normal-turn argument transparent, but the same strategy may extend more broadly, for instance to hyperbolic one-holed tori and cone tori. There are also several generalizations of the Markoff
equation in literature. For instance,
Gyoda--Matsushita \cite{GyodaMatsushita} studied
$$
x^2+y^2+z^2+k_1yz+k_2zx+k_3xy
=
(3+k_1+k_2+k_3)xyz,
$$
which contains the symmetric $k$-Markoff equation when
$k_1=k_2=k_3=k$. Related generalized Markoff numbers and orderings were studied
by Banaian--Sen and Banaian \cite{BanaianSen,BanaianOrderings}. For these generalized Markoff equations, our method suggests a
colored shifted analog, and the normal-turn method should provide the
polylogarithmic fiber bounds and a local-boundary picture.

A different direction is to move beyond the once-punctured torus. Let
$S=S_{g,n}$ be a finite-type surface with
$3g-3+n\geq1,$
and let $X$ be a hyperbolic structure on $S$. This includes the four-holed
sphere, as well as surfaces of higher complexity. Simple closed curves are then no
longer parametrized by a two-dimensional lattice. The natural replacement is
the length ball
$$
\{\lambda\in\mathcal{ML}(S):\ell_X(\lambda)\leq 1\},
$$
together with the integral multicurves and Thurston measure. This is the
setting of Mirzakhani's counting theorem and the work of
Erlandsson--Souto on geodesic currents
\cite{Mirzakhani,ErlandssonSoutoCounting,ErlandssonSoutoBook}. A useful fact is the convexity of hyperbolic length in train-track
coordinates, which follows from Mirzakhani's convexity theorem
\cite[Theorem~A.1]{MirzakhaniThesis}; see also
\cite{eskin2022effective}. This makes it possible to study equal-length curves
using higher-dimensional convex geometry. A first estimate based on Andrews's
theorem for lattice polytopes \cite{andrews1963lower} gives
$$
m_X(L)\leq C_XL^{D(D-1)/(D+1)},
$$
where
$
D=\dim\mathcal{ML}(S)=6g-6+2n.
$
In forthcoming work, we aim to extend these methods to obtain sharper bounds for simple-length multiplicities and to develop a local geometric description of the boundary of the length ball.

\bibliographystyle{alpha}
\bibliography{references}

\end{document}